\documentclass[11pt]{amsart}
\usepackage[a4paper,margin=29mm]{geometry}
\usepackage[T1]{fontenc}
\usepackage{lmodern}
\usepackage{microtype}
\usepackage{indentfirst}
\usepackage{amsmath,amssymb,amsthm,mathtools}
\usepackage{xcolor}
\usepackage[
  colorlinks=false,
  pdfborder={0 0 1},
  linkbordercolor=red,
  citebordercolor=green,
  urlbordercolor=blue
]{hyperref}

\usepackage{etoolbox}

\makeatletter

\patchcmd{\thebibliography}
  {\usecounter{enumiv}}
  {\usecounter{enumiv}%
   \setlength{\itemsep}{4pt}%
   \setlength{\parsep}{0pt}}
  {}{}

\makeatother

\allowdisplaybreaks
\newtheorem{theorem}{Theorem}[section]
\newtheorem{lemma}[theorem]{Lemma}
\newtheorem{cor}[theorem]{Corollary}
\newtheorem{proposition}[theorem]{Proposition}
\newtheorem{definition}[theorem]{Definition}
\theoremstyle{remark}
\newtheorem{remark}[theorem]{Remark}
\theoremstyle{plain}
\newtheorem{problem}[theorem]{Problem}
\newcommand{\cF}{\mathcal F}
\newcommand{\Gm}{G_{\mathbf m}}

\newcommand{\N}{\mathbb N}

\begin{document}

 \baselineskip 16.6pt
\hfuzz=6pt

\widowpenalty=10000

\renewcommand{\theequation}
{\thesection.\arabic{equation}}

\numberwithin{equation}{section}

\title[Vilenkin Partial Sums]
{Vector-valued partial sums on unbounded Vilenkin systems}

\author{Deyu Chen}
\address{Deyu Chen, Institute for Advanced Study in Mathematics,
Harbin Institute of Technology, Harbin 150001, China}
\email{1201200317@stu.hit.edu.cn}

\author{Guixiang Hong}
\address{Guixiang Hong, Institute for Advanced Study in Mathematics,
Harbin Institute of Technology, Harbin 150001, China}
\email{gxhong@hit.edu.cn}

\date{}



\begin{abstract}
Let \(\Gm=\prod_{k\ge0}\mathbb Z_{m_k}\) be a Vilenkin
group that is not necessarily bounded, i.e.,
\(\sup_k m_k=\infty\). We prove that, for every UMD Banach space
\(X\) and every \(1<p<\infty\), the Vilenkin partial-sum operators are
uniformly bounded on \(L^p(\Gm;X)\), with a bound depending only on
\(p\) and the UMD constant of \(X\), and not on \(\mathbf m\). This resolves an open problem arising from the work of Clément et al.~\cite{ClementDePagterSukochevWitvliet2000} and later recorded explicitly in the book of Hytönen et al.~\cite[p.~362]{HNVWI}. The proof reduces the partial-sum estimate, via a Paley conjugation identity and a tangent-sequence decoupling argument, to a decoupling inequality for Fourier projections on finite cyclic groups, which appears to be new.
The same approach also yields \(\mathcal R\)-boundedness for the family of partial-sum operators associated with the finer block decomposition, thereby resolving another related problem communicated to us by Fedor Sukochev.
\end{abstract}

\maketitle

\section{Introduction}
\subsection{Background}
Let \(\mathbf m=(m_k)_{k\ge0}\), where \(m_k\ge2\), and let
\[
\Gm:=\prod_{k\ge0}\mathbb Z_{m_k}
\]
be the associated Vilenkin group, equipped with its normalized Haar
measure \(\mu\). The Vilenkin system is called bounded
if \(\sup_km_k<\infty\) and unbounded otherwise.
Let \((\psi_n)_{n\ge0}\) be the corresponding Vilenkin character
system. For an \(X\)-valued Vilenkin polynomial \(f\), set
\[
\widehat f(n)
:=
\int_{\Gm}f(x)\overline{\psi_n(x)}\,d\mu(x).
\]
If \(\phi:\mathbb N\to\mathbb C\), the associated
Fourier--Vilenkin multiplier is formally given by
\[
T_\phi f
:=
\sum_{n\ge0}\phi(n)\widehat f(n)\psi_n.
\]
For \(A\subseteq\mathbb N\), we write \(P_A:=T_{\mathbf1_A}\). In
particular,
\[
S_n:=P_{[0,n)},\quad n\ge1,
\qquad S_0:=0,
\]
is the \(n\)-th Vilenkin partial-sum operator.

The scalar theory of Fourier--Vilenkin series has a long history. A central question is the uniform boundedness of the partial-sum operators, which is closely related to the mean convergence of Vilenkin--Fourier series. Early results were obtained only for bounded Vilenkin systems; see, for instance, \cite{Watari1958,Gosselin1973}. Significant progress was made by Young~\cite{Young1976}, who established uniform weak $(1,1)$ and strong $(p,p)$ boundedness of partial sums for arbitrary Vilenkin systems, including the unbounded case. In subsequent work, Young further developed weighted estimates and multiplier theorems in this unbounded setting~\cite{Young1993,Young1994}. These results constituted a genuine extension of the classical bounded theory to arbitrary Vilenkin systems. For more details on the classical theory, we refer to the monograph of Schipp, Wade, and Simon~\cite{SchippWadeSimon1990}.

For vector-valued functions, Cl{\'e}ment, de Pagter, Sukochev, and
Witvliet~\cite{ClementDePagterSukochevWitvliet2000} proved that, for a
fixed bounded Vilenkin system and \(1<p<\infty\), the frequency
subspaces \((\psi_nX)_{n\ge0}\) form a Schauder decomposition of
\(L^p(\Gm;X)\) if and only if \(X\) is UMD. Equivalently, for each
fixed bounded Vilenkin system, \(X\) is UMD if and only if
$$
\sup_{n\ge1}
\|S_n\|_{L^p(\Gm;X)\to L^p(\Gm;X)}<\infty,
$$
where the bound may depend on \(\sup_k m_k\). They left open whether the same characterization remains valid for arbitrary, possibly unbounded, Vilenkin systems. This problem
was later recorded explicitly in the book of Hyt\"onen, van Neerven, Veraar, and
Weis~\cite[p.~362]{HNVWI}, which can be formulated as follows.

\begin{problem}\label{p1}
Let \(1<p<\infty\) and let \(X\) be a \(\operatorname{UMD}\) Banach
space. Does there exist a constant \(C_{p,X}\), depending only on \(p\)
and the \(\operatorname{UMD}\) constant of \(X\), such that, for every
generating sequence \(\mathbf m=(m_k)_{k\ge0}\) with \(m_k\ge2\),
$$
\sup_{n\ge1}
\|S_n\|_{L^p(\Gm;X)\to L^p(\Gm;X)}
\le C_{p,X}?
$$
\end{problem}

This problem is considerably more difficult because the uniform bound is required to be independent of the generating sequence \(\mathbf m\).
To the best of our knowledge, previous progress was limited to the following special case. Using a generalized triangular truncation argument, Dodds, Ferleger, de Pagter, and Sukochev~\cite{DoddsFerlegerDePagterSukochev2001} proved that, for every semifinite von Neumann algebra \(\mathcal M\) and every \(1<p<\infty\),
$$
\sup_{n\ge1}
\|S_n\|_{L^p(\Gm;L^p(\mathcal M))
\to L^p(\Gm;L^p(\mathcal M))}
\le C_p.
$$
 In particular, this gives the Schauder-decomposition property for arbitrary Vilenkin systems in noncommutative \(L^p\)-spaces.
More recently, Hong and Zhao~\cite{HongZhao2026} developed a variant of the noncommutative Calder\'on--Zygmund argument to obtain the uniform weak type \((1,1)\) estimate
$$
\sup_{n\ge1}
\|S_n\|_{L_1(\mathcal N)\to L_{1,\infty}(\mathcal N)}
\le C,
\qquad
\mathcal N=L^\infty(\Gm)\,\overline\otimes\,\mathcal M,
$$
for arbitrary Vilenkin systems and semifinite von Neumann algebras \(\mathcal M\), and thus provided the sharp order $C_p\simeq \frac{p^2}{p-1}$. 
 For related results, see also \cite{ScheckterSukochev2018,TianZhou2024} and the references therein.

\subsection{The first main result and proof strategy}

Our main theorem gives an affirmative answer to Problem~\ref{p1}.
\begin{theorem}
\label{thm:uniform-umd-vilenkin-partial-sums}
Let \(X\) be a \(\operatorname{UMD}\) Banach space with \(\operatorname{UMD}\) constant $\beta_{p,X}$,  and let \(1<p<\infty\). There exists a
constant \(C_{p}<\infty\) depending only on \(p\) such that, for every generating sequence
\(\mathbf m=(m_k)_{k\ge0}\) with \(m_k\ge2\) and
for every \(f\in L^p(\Gm;X)\),
\begin{equation}\label{eq:uniform-umd-partial-sums}
\sup_{n\geq1}\|S_nf\|_{L^p(\Gm;X)}
\le C_{p}\beta_{p,X}^4\|f\|_{L^p(\Gm;X)}.
\end{equation}
\end{theorem}

As a consequence, we obtain a characterization of the UMD property
valid for every, possibly unbounded, Vilenkin system.
\begin{cor}
\label{cor:umd-partial-sums-schauder}
Let \(X\) be a Banach space, let \(1<p<\infty\) and fix an arbitrary
generating sequence \(\mathbf m=(m_k)_{k\ge0}\), \(m_k\ge2\). The
following assertions are equivalent:
\begin{enumerate}
\item[\rm{(i)}] \(X\) is a \(\operatorname{UMD}\) Banach space.
\item[\rm{(ii)}] The Vilenkin partial-sum operators are uniformly bounded in
\(n\), that is,
\begin{equation}\label{eq:partial-sum-uniform-corrected}
\sup_{n\ge1}
\|S_n\|_{L^p(\Gm;X)\to L^p(\Gm;X)}<\infty.
\end{equation}
\item[\rm{(iii)}] The sequence of frequency subspaces
\[
\bigl(\psi_nX\bigr)_{n\ge0},
\qquad
\psi_nX:=\{\psi_nx:x\in X\},
\]
is a Schauder decomposition of \(L^p(\Gm;X)\).
\end{enumerate}
Moreover, if one of these assertions holds, then the bound in
\eqref{eq:partial-sum-uniform-corrected} can be chosen independently of
the generating sequence \(\mathbf m\), as in
\eqref{eq:uniform-umd-partial-sums}.
\end{cor}


\begin{remark}
In addition to the easy consequence---Corollary \ref{cor:umd-partial-sums-schauder}, Theorem \ref{thm:uniform-umd-vilenkin-partial-sums} will play a substantial and nontrivial role in a subsequent work on the vector-valued Rubio de Francia-Littlewood-Paley square function inequality on Vilenkin systems \cite{ChenHong}. On the other hand, it would be interesting to determine the optimal order of $\beta_{p,X}$ in \eqref{eq:uniform-umd-partial-sums}.
\end{remark}

Since systems are allowed to be unbounded, the proof of Theorem \ref{thm:uniform-umd-vilenkin-partial-sums} requires ideas that are substantially different from those used in the bounded case. To explain the main difficulty and our approach, we first introduce some necessary notation and briefly recall the key argument for bounded Vilenkin systems. Given a generating sequence
\(\mathbf m=(m_k)_{k\ge0}\), define
\begin{equation}\label{Mk}
M_0:=1,
\qquad
M_{k+1}:=m_kM_k,
\qquad k\ge0.
\end{equation}
Let
\(
\Lambda_{\mathbf m}
:=
\bigl\{(k,\ell):k\ge0,\ 1\le\ell<m_k\bigr\}
\)
and for \((k,\ell)\in\Lambda_{\mathbf m}\), set
\[
\delta_{k,\ell}
:=
[\ell M_k,(\ell+1)M_k),
\qquad
\Delta_{k,\ell}
:=
P_{\delta_{k,\ell}}.
\]
Write \(n=\sum_{k=0}^K n_kM_k\), where \(0\le n_k<m_k\) and $K$ is the largest number $k\in\mathbb N$ such that $n_k\ne 0$. The Paley conjugation identity (Lemma~\ref{lem:paley-conjugation-identity}) allows us to rewrite the $n$-th partial-sum operator as
\begin{align*}
    S_nf&{=}\psi_n\sum_{0\le k\le K}\sum_{\ell=m_k-n_k}^{m_k-1}\Delta_{k,\ell}(\overline{\psi_n}f)=\psi_n\sum_{\ell=1}^{\sup_km_k-1}\sum_{\substack{0\le k\le K\\ m_k-n_k\le \ell\le m_k-1}}\Delta_{k,\ell}(\overline{\psi_n}f).
\end{align*}
By the triangle inequality, the outer sum in \(\ell\) may be moved outside the norm,
\begin{align*}
 \|S_nf\|_{L^p(\Gm;X)}\le \sum_{\ell=1}^{\sup_km_k-1}\left\|\sum_{\substack{0\le k\le K\\ m_k-n_k\le \ell\le m_k-1}}\Delta_{k,\ell}(\overline{\psi_n}f)\right\|_{L^p(\Gm;X)}.
 \end{align*}
As follows from \eqref{finer}, for a fixed $\ell$, the functions $\Delta_{k,\ell}f$, with $m_k>\ell$, form a martingale difference sequence. Hence, the vector-valued Stein inequality, followed by two applications of the UMD inequality, controls each inner norm by a constant multiple of $\|f\|_{L^p(\Gm;X)}$. More details can be found in Section \ref{sec:proof of bounded}.

Obviously, the above argument is specific to bounded Vilenkin systems. To deal with the unbounded case, we use the notation in \eqref{eq:pi_k,a}, namely,
$$\pi_{k,n_k}=\sum_{\ell=m_k-n_k}^{m_k-1}\Delta_{k,\ell},$$
and estimate the following expression 
\begin{align}\label{pi}S_nf=\psi_n\sum_{k=0}^K\pi_{k,n_k}d_k(\overline{\psi_n}f)\end{align}
directly. We achieve this by combining several decoupling arguments with a sampling factorization identity; moreover, we discover several equivalent characterizations of the UMD property in terms of decoupling inequalities that might be of independent interest. The one playing a decisive role is called the {\it cyclic projections decoupling property}, abbreviated as CPD, which is related to the first Hoeffding chaos; see Definition \ref{def:CPD-property} and Theorem \ref{thm:umd-implies-CPD}. With this property, one may then exploit McConnell's
decoupling inequality for tangent martingale-difference sequences in UMD spaces to handle \eqref{pi} in Proposition \ref{prop:adapted-terminal-core}, and thus complete the proof of the uniform boundedness of the partial-sum operators.

The proof of the equivalence UMD$\Longleftrightarrow$CPD consists of four intermediate results. In the first step, we establish a decoupling inequality \eqref{eq:diagonal-bv-independent} for Marcinkiewicz multipliers on the tori in Theorem \ref{thm:diagonal-bv-independent}. In the second step, we obtain a decoupling inequality \eqref{eq:decoupling-cyclic-convolution} for convolution operators on the cyclic groups in Proposition~\ref{prop:inverse-sinc-uniform-kernel}. Both steps use techniques from Bourgain~\cite{Bourgain1983}. In the third step, we establish the sampling factorization identity \eqref{eq:exact-cyclic-sampling-factorization} in Lemma~\ref{lem:exact-cyclic-sampling-factorization}, which links Fourier projections on finite cyclic groups to Marcinkiewicz multipliers on the torus. In the fourth step, we exploit the two decoupling inequalities \eqref{eq:diagonal-bv-independent}, \eqref{eq:decoupling-cyclic-convolution} and the sampling identity \eqref{eq:exact-cyclic-sampling-factorization} to derive a decoupling inequality for centered cyclic projections \eqref{eq:independent-centered-cyclic-intervals} in Proposition \ref{prop:independent-centered-cyclic-intervals}. Finally, the passage from centered projections to terminal projections in the desired inequality \eqref{eq:CPD-property} is straightforward. This gives UMD$\Rightarrow$CPD; for the converse direction, we link the terminal projection with the Riesz projection through approximation; see the end of Section \ref{s5.3}. These implications also show that each of \eqref{eq:diagonal-bv-independent} and \eqref{eq:independent-centered-cyclic-intervals} characterizes the UMD property of $X$.

We finally point out that McConnell's decoupling inequality, originally employed in the construction of It\^o-type integrals for UMD-valued stochastic processes, has previously been used by Hyt\"onen in vector-valued Euclidean harmonic analysis \cite{Hytonen-IMRN,HytonenPseudo2011}. Here we apply this method to vector-valued Vilenkin Fourier analysis. A common feature of these applications is the presence of non-doubling measures or irregular filtrations.

\subsection{The second main result: \texorpdfstring{$\mathcal R$}{R}-boundedness of partial-sum projections associated to fine blocking}

The fine blocking and its associated partial-sum projections provide
another connection between the UMD property and Vilenkin systems. We first introduce some additional notation.
We equip \(\Lambda_{\mathbf m}\) with the lexicographic order, so that
\[
(r,i)<(k,j)
\quad\Longleftrightarrow\quad
r<k
\ \text{or}\
\bigl(r=k\ \text{and}\ i<j\bigr).
\]
For \((k,j)\in\Lambda_{\mathbf m}\), define the corresponding
fine-block partial-sum projection, with the constant block omitted, by
\[
S_{k,j}
:=
\sum_{\substack{(r,i)\in\Lambda_{\mathbf m}\\(r,i)<(k,j)}}
\Delta_{r,i}.
\]

In the bounded case, Cl\'ement et al.~\cite[Lemma~4.3]{ClementDePagterSukochevWitvliet2000}
proved that, for every UMD Banach space \(X\), the family
\(
\{S_{k,j}:(k,j)\in\Lambda_{\mathbf m}\}
\)
is \(\mathcal R\)-bounded on \(L^p(\Gm;X)\), with a bound that may
depend on \(\sup_km_k\). Combining this result with $\sup_km_k<\infty$ and the unconditionality
of the coarse martingale-difference decomposition, they further showed that the fine-block decomposition, with the
constant block adjoined, is unconditional in \(L^p(\Gm;X)\); see
\cite[Corollary~4.4]{ClementDePagterSukochevWitvliet2000}.

In the unbounded case, however, Watari's
counterexamples~\cite[Theorem~4]{Watari1958} show that this
fine-block unconditionality fails already in scalar-valued
\(L^p(\Gm)\) whenever \(p\ne2\). Nevertheless, the associated
ordered partial-sum projections remain \(\mathcal R\)-bounded for
arbitrary Vilenkin systems. This resolves another related problem communicated to us by Fedor Sukochev. The proof is obtained by a Rademacher-valued version of the argument used for
Theorem~\ref{thm:uniform-umd-vilenkin-partial-sums}.

\begin{theorem}\label{thm:R-bound}
Let \(X\) be a \(\operatorname{UMD}\) Banach space and  \(1<p<\infty\). Then the family
\(
\{S_{k,j}:(k,j)\in\Lambda_{\mathbf m}\}
\)
is \(\mathcal R\)-bounded on \(L^p(\Gm;X)\). More precisely, there
exists a constant \(C_{p}<\infty\), depending only on \(p\), such that, for every generating sequence $\mathbf m=(m_k)_k$ with $m_k\ge 2,$ every \(N\ge1\) and every \((k_a,j_a)\in\Lambda_{\mathbf m}\), 
\(h_a\in L^p(\Gm;X)\), \(1\le a\le N\),
\begin{equation}\label{eq:R-bound}
\left\|
\sum_{a=1}^N
\varepsilon_aS_{k_a,j_a}h_a
\right\|_{L^p(\Omega_\varepsilon\times\Gm;X)}
\le C_p
\beta_{p,X}^5
\left\|
\sum_{a=1}^N
\varepsilon_ah_a
\right\|_{L^p(\Omega_\varepsilon\times\Gm;X)}.
\end{equation}

Conversely, suppose that
\(\sup_km_k=\infty\) and \(X\ne\{0\}\). If the family
\(\{S_{k,j}:(k,j)\in\Lambda_{\mathbf m}\}\) is
\(\mathcal R\)-bounded on \(L^p(\Gm;X)\), then \(X\) is \(\operatorname{UMD}\).
\end{theorem}
\begin{remark}
In Theorem~\ref{thm:R-bound}, the converse implication requires the
unboundedness assumption
\(
\sup_k m_k=\infty.
\)
When \(\sup_k m_k<\infty\), we do not know whether the converse
implication remains valid. Indeed, the \(\mathcal R\)-boundedness of the
family
\(
\{S_{k,j}:(k,j)\in\Lambda_{\mathbf m}\}
\)
is equivalent to the corresponding vector-valued Stein inequality for
the canonical Vilenkin filtration. Thus, the converse problem is closely connected with the long-standing question of whether the vector-valued Stein inequality can hold in a class of Banach spaces strictly larger than the class of UMD spaces; see
\cite[Remark~6.3(ii)]{VeraarWeis2010}.
\end{remark}

By definition, \(S_{k,j}+E_0=S_{jM_k}\). Consequently, the family
\(
\{S_{k,j}:(k,j)\in\Lambda_{\mathbf m}\}
\)
is \(\mathcal R\)-bounded if and only if
\(
\{S_{jM_k}:(k,j)\in\Lambda_{\mathbf m}\}
\)
is \(\mathcal R\)-bounded. It is therefore natural to ask whether the
full family of partial-sum operators,
\(
\{S_n:n\in\mathbb N\},
\)
is \(\mathcal R\)-bounded on \(L^p(\Gm;X)\) for every \(1<p<\infty\) and every UMD Banach space \(X\). In the scalar case \(X=\mathbb C\), 
this was proved by Young~\cite{Young1990}. Cl\'{e}ment et al.~\cite[Theorem~5.1]{ClementDePagterSukochevWitvliet2000}
proved, for bounded Vilenkin systems, that the full partial-sum family is \(\mathcal R\)-bounded whenever \(X\) is a $\operatorname{UMD}$
space with property \((\alpha)\). The proof relies on the Paley conjugation identity and the unconditionality of fine-block decomposition. We extend this result to arbitrary Vilenkin systems and, at the same time, obtain a converse: the full family of partial-sum operators is \(\mathcal R\)-bounded on \(L^p(\Gm;X)\) if and only if \(X\) is UMD and has property \((\alpha)\).

\begin{proposition}\label{prop:R-bounded-implies-alpha}
Let \(1<p<\infty\) and let \(X\) be a Banach space. Fix an arbitrary
generating sequence \(\mathbf m\). Then
\(\{S_n:n\in\mathbb N\}\) is \(\mathcal R\)-bounded on
\(L^p(\Gm;X)\) if and only if \(X\) is $\operatorname{UMD}$ and has property \((\alpha)\).
\end{proposition}



\subsection{Notation}

Unless otherwise stated, we only consider complex Banach spaces.
Throughout, \(\mathbb N:=\{0,1,2,\ldots\}\), and
\(\mathbb T:=\mathbb R/\mathbb Z\) is equipped with normalized Haar
measure. For \(m\ge2\), the cyclic group
\(\mathbb Z_m:=\mathbb Z/m\mathbb Z\) is equipped with normalized
counting measure. For \(a<b\), \([a,b)\) denotes the corresponding
half-open interval, interpreted as a subset of \(\mathbb Z\) or
\(\mathbb N\) when the endpoints are integers. The symbol
\(\mathbf1_E\) denotes the indicator of \(E\), and
\(\operatorname{supp}f\) denotes the support of \(f\). For \(1<p<\infty\), we use \(p'=p/(p-1)\) to denote the conjugate exponent.
For nonnegative quantities \(A\) and \(B\), the notation
\(A\lesssim_\Theta B\) means that
\(A\le C_\Theta B\), where \(C_\Theta\) depends only on the parameters
listed in \(\Theta\). The notation \(A\simeq_\Theta B\) means that both
\(A\lesssim_\Theta B\) and \(B\lesssim_\Theta A\) hold.

We write \(L^p(\Omega;X)\) for the Bochner space of \(X\)-valued
functions, and \(L_0^p(\Omega;X)\) for its mean-zero subspace when
\(\Omega\) is a probability space. The symbol \(\mathcal B(X)\) denotes
the Borel \(\sigma\)-algebra of \(X\). Expectations and probabilities
are denoted by \(\mathbb E\) and \(\mathbb P\), with subscripts used to
indicate the relevant variables. An auxiliary Rademacher sequence is
denoted by \((\varepsilon_k)_k\), and its underlying probability space
is denoted by \(\Omega_\varepsilon\).

\section{Preliminaries}\label{sec:preliminaries}

We begin with the Vilenkin characters, the canonical filtration, and
the associated frequency projections. We then introduce the Paley
conjugation identity and the Banach-space notions used in the
operator estimates.

\subsection{Basic definitions for Vilenkin systems}\label{s2.1}
Let $M_k$ be defined as in \eqref{Mk}.
Every \(n\in\N\) has a unique mixed-radix expansion
\[
n=\sum_{k\ge0}n_kM_k,
\qquad
0\le n_k<m_k,
\]
with only finitely many nonzero digits. We call \(n_k\) the \(k\)-th
digit of \(n\). For \(a,b\in\mathbb N\), define digitwise addition and
subtraction by
\[
a\oplus b
:=
\sum_{k\ge0}\bigl((a_k+b_k)\bmod m_k\bigr)M_k,
\qquad
a\ominus b
:=
\sum_{k\ge0}\bigl((a_k-b_k)\bmod m_k\bigr)M_k.
\]
Thus \((\mathbb N,\oplus)\) is naturally identified with the dual
group of \(\Gm\).

For \(x=(x_k)_{k\ge0}\in\Gm\), define
\[
r_k(x):=e^{2\pi ix_k/m_k},
\qquad
\psi_n(x):=\prod_{k\ge0}r_k(x)^{n_k}.
\]
Then \((\psi_n)_{n\ge0}\) is the Vilenkin character system. In the
Walsh case, \(m_k=2\) for every \(k\), and
\[
r_k=\psi_{2^k},
\qquad
r_k(x)\in\{-1,1\};
\]
these are the classical Rademacher functions.
For an \(X\)-valued Vilenkin polynomial \(f\), we write
\[
\widehat f(n)
:=
\int_{\Gm}f(x)\overline{\psi_n(x)}\,d\mu(x),
\qquad
P_Af
:=
\sum_{n\in A}\widehat f(n)\psi_n.
\]
The character identities
\[
\psi_a\psi_b=\psi_{a\oplus b},
\qquad
\overline{\psi_a}=\psi_{0\ominus a}
\]
give the modulation identity
\begin{equation}\label{eq:modulation-identity}
P_{c\oplus A}f
=
\psi_cP_A(\overline{\psi_c}f),
\qquad
c\in\mathbb N,\qquad A\subseteq\mathbb N.
\end{equation}

We now introduce the canonical filtration on the Vilenkin system. Let
\[
\cF_k:=\sigma(x_0,\ldots,x_{k-1}),
\qquad k\ge1,
\]
and let \(\cF_0\) be the trivial \(\sigma\)-algebra. The corresponding
conditional expectation is
\(
E_k=P_{[0,M_k)}.
\)
The finer blocking is defined by
\[
\delta_{k,\ell}:=[\ell M_k,(\ell+1)M_k),
\qquad
\Delta_{k,\ell}:=P_{\delta_{k,\ell}},\qquad k\ge0,\quad 1\le\ell<m_k.
\]
The \(k\)-th martingale-difference operator is therefore
\begin{equation*}
d_k:=E_{k+1}-E_k
=
\sum_{\ell=1}^{m_k-1}\Delta_{k,\ell},
\qquad k\ge0.
\end{equation*}
In \cite{ClementDePagterSukochevWitvliet2000}, $(d_k)_k$ is called
the coarse blocking, in contrast to the finer blocking
$\{\Delta_{k,\ell}\}_{k,\ell}$.

Since
\(\delta_{k,\ell}=\ell M_k\oplus[0,M_k)\),
\eqref{eq:modulation-identity} gives
\begin{align}\label{finer}
    \Delta_{k,\ell}f= \Delta_{k,\ell}d_kf
=
\psi_{\ell M_k}E_k(\overline{\psi_{\ell M_k}}d_kf).
\end{align}

\subsection{Paley conjugation identity}\label{s2.2}
We now introduce the Paley conjugation identity. It appears in the
Schauder-decomposition framework of Cl\'ement, de Pagter, Sukochev, and
Witvliet~\cite[Section~4]{ClementDePagterSukochevWitvliet2000}; see also
Young~\cite[pp.~312--313]{Young1976} for the corresponding Dirichlet
kernel decomposition. For
\[
n
=
\sum_{k=0}^Kn_kM_k,
\qquad
0\le n_k<m_k,
\]
define
\begin{equation}\label{eq:paley-terminal-set}
B_n
:=
\bigcup_{\substack{0\le k\le K\\n_k\ne0}}
\ \bigcup_{\ell=m_k-n_k}^{m_k-1}
\delta_{k,\ell}.
\end{equation}
Define the projection
\begin{equation}\label{eq:pi_k,a}
\pi_{k,a}
:=
\sum_{\ell=m_k-a}^{m_k-1}\Delta_{k,\ell},
\qquad 0\le a<m_k,\qquad
\pi_{k,0}:=0.
\end{equation}

\begin{lemma}[Paley conjugation identity]
\label{lem:paley-conjugation-identity}
For every \(\gamma\in\mathbb N\),
\begin{equation}\label{eq:paley-digit-criterion}
0\le\gamma<n
\quad\Longleftrightarrow\quad
\gamma\ominus n\in B_n.
\end{equation}
Consequently,
\begin{equation}\label{eq:paley-conjugation-operator}
S_nf
=
\psi_nP_{B_n} 
\bigl(\overline{\psi_n}f\bigr),
\end{equation}
and
\begin{equation}\label{eq:paley-terminal-martingale-decomposition}
P_{B_n}f
=\sum_{k=0}^K
\pi_{k,n_k}f=
\sum_{k=0}^K
\pi_{k,n_k}d_kf.
\end{equation}
\end{lemma}

\begin{proof}
We give a short proof for completeness.
If \(\gamma=n\), then
\(\gamma\ominus n=0\notin B_n\), so both sides of
\eqref{eq:paley-digit-criterion} are false. Suppose that
\(\gamma\ne n\), and let \(k\) be the largest index for which
\(\gamma_k\ne n_k\). Then
\[
\gamma<n
\quad\Longleftrightarrow\quad
\gamma_k<n_k.
\]
In this case, the largest nonzero digit of \(\gamma\ominus n\) is the
\(k\)-th digit, and its value is
\[
m_k+\gamma_k-n_k
\in
\{m_k-n_k,\ldots,m_k-1\}.
\]
Hence \(\gamma\ominus n\in B_n\).

Conversely, suppose that the largest nonzero digit of
\(\gamma\ominus n\) is the \(k\)-th digit and that
\(\gamma\ominus n\in\delta_{k,\ell}\) for one of the intervals occurring
in \(B_n\). We must then have
\(\gamma_k<n_k\). Otherwise,
\[
0
\le
\gamma_k-n_k
<
m_k-n_k,
\]
which would imply
\(\gamma\ominus n\notin\delta_{k,\ell}\), a contradiction. Therefore,
\(\gamma<n\).

The criterion is equivalent to
\[
[0,n)=n\oplus B_n.
\]
Thus \eqref{eq:paley-conjugation-operator} follows from the modulation
identity \eqref{eq:modulation-identity}. Finally, the blocks in
\eqref{eq:paley-terminal-set} are pairwise disjoint, each lies in the
Fourier support of \(d_k\), and the Fourier--Vilenkin projection
associated with
\[
\bigcup_{\ell=m_k-n_k}^{m_k-1}
\delta_{k,\ell}
\]
is precisely \(\pi_{k,n_k}\). This proves
\eqref{eq:paley-terminal-martingale-decomposition}.
\end{proof}

\subsection{UMD property and \texorpdfstring{$\mathcal R$}{R}-boundedness}
\label{sec:umd-and-r-boundedness}
We recall the notions of the UMD property and $\mathcal R$-boundedness, and the properties needed below. Standard
references are the monographs of Pisier~\cite{Pisier2016} and
Hyt\"onen et~al.~\cite{HNVWI,HNVWII}.

\begin{definition}[\(\operatorname{UMD}_p\) property]
Let \(1<p<\infty\). A Banach space \(X\) is said to have the
\emph{\(\operatorname{UMD}_p\) property} if there exists a constant
\(C<\infty\) such that
\begin{equation}\label{UMD}
    \left\|
\sum_{k=1}^n\theta_kd_k
\right\|_{L^p(\Omega;X)}
\le
C
\left\|
\sum_{k=1}^nd_k
\right\|_{L^p(\Omega;X)}
\end{equation}
for every filtered probability space, every finite \(X\)-valued
martingale difference sequence \((d_k)_{k=1}^n\), and every choice of
signs \(\theta_k\in\{-1,1\}\). The infimum of $C$ is denoted by $\beta_{p,X}.$
\end{definition}


The definition of the UMD$_p$ property is independent of the exponent $p,$ which is recorded as the following proposition.

\begin{proposition}\label{prop:UMD}
  The $\operatorname{UMD}_p$ property does not depend on the choice of exponent $p.$ Equivalently, if \eqref{UMD} holds for one
\(p\in(1,\infty)\), then it holds for every \(p\in(1,\infty)\).
\end{proposition}

\begin{remark}\label{Riesz}
From the above proposition, a Banach space $X$ is said to have UMD property if it satisfies \eqref{UMD}.
    There is an important characterization of the UMD property: A Banach space $X$ is \(\operatorname{UMD}\) if and only if the Riesz projection $P_+^{\mathbb T}$ is bounded on $L^p(\mathbb T;X)$ for some (equivalently for every) $p\in(1,\infty)$, where $P_+^{\mathbb T}$ is defined by
    $$P_+^{\mathbb T}\left(\sum_{n\in\mathbb Z}x_ne^{2\pi in\cdot}\right)=\sum_{n\ge 0}x_ne^{2\pi in\cdot}$$ for any $X$-valued trigonometric polynomial $\sum_{n\in\mathbb Z}x_ne^{2\pi in\cdot}.$ Moreover, we have $$\beta_{p,X}^{1/2}\lesssim_p\|P_+^{\mathbb T}\|_{L^p(\mathbb T;X)\to L^p(\mathbb T;X)}\lesssim_p \beta_{p,X}^2.$$ See \cite{Bourgain1983,HNVWI}. A recent work \cite{LoristVanNeerven2026} shows that the above dependence is sharp.
\end{remark}

We also use the following fact; see e.g. \cite[Example 4.3]{DirksenMaasVanNeerven2013}.
\begin{lemma}\label{Bochner1}
   Let $X$ be a Banach space with the $\operatorname{UMD}$ property. For any $1<p<\infty$ and any measure space $(S,\mu),$ $L^p(S;X)$ has the $\operatorname{UMD}$ property with $\beta_{p,L^p(S;X)}\le \beta_{p,X}.$
\end{lemma}

Now we introduce $\mathcal R$-boundedness for a family of operators.

\begin{definition}[$\mathcal R$-boundedness]
Let \(X\) and \(Y\) be Banach spaces, and let
\(\mathcal L(X,Y)\) denote the space of bounded linear operators from
\(X\) to \(Y\). A family
\[
\mathcal T=\{T_j:j\in\Lambda\}\subseteq\mathcal L(X,Y)
\]
is said to be \emph{\(\mathcal R\)-bounded} if, for some
\(1\le q<\infty\), there exists a constant \(C<\infty\) such that
\begin{equation}\label{def:R-bound}
\left\|
\sum_{a=1}^N\varepsilon_aT_{j_a}x_a
\right\|_{L^q(\Omega_\varepsilon;Y)}
\le
C
\left\|
\sum_{a=1}^N\varepsilon_ax_a
\right\|_{L^q(\Omega_\varepsilon;X)}
\end{equation}
for every \(N\ge1\), every \(j_1,\ldots,j_N\in\Lambda\), and every
\(x_1,\ldots,x_N\in X\). For a fixed exponent \(q\), the infimum of $C$ is denoted by $\mathcal R_q(\mathcal T).$
\end{definition}

We shall use the following two basic properties of  \(\mathcal R\)-boundedness.
\begin{proposition}\label{lem:R}
Let \(X\) and \(Y\) be Banach spaces, and let
\(\mathcal T\subseteq\mathcal L(X,Y)\).

\begin{enumerate}
\item[(i)]
Suppose that, for some \(1\le q_0<\infty\), $\mathcal R_{q_0}(\mathcal T)<\infty$. Then, for any \(1\le q<\infty\), $$\mathcal R_q(\mathcal T)\simeq_{q,q_0} \mathcal R_{q_0}(\mathcal T).$$
Consequently, the definition of $\mathcal R$-boundedness is independent
of the Rademacher exponent.

\item[(ii)]
Define the absolutely convex hull of \(\mathcal T\) by
\[
\operatorname{aconv}(\mathcal T)
:=
\left\{
\sum_{j=1}^M\alpha_jT_j:
M\ge1,\ T_j\in\mathcal T,\ 
\sum_{j=1}^M|\alpha_j|\le1
\right\}.
\]
If \(\mathcal T\) is $\mathcal R$-bounded, then
\(\operatorname{aconv}(\mathcal T)\) is $\mathcal R$-bounded with
\[
\mathcal R_q\bigl(\operatorname{aconv}(\mathcal T)\bigr)
\le 2
\mathcal R_q(\mathcal T).
\]
\end{enumerate}
\end{proposition}

\subsection{Uniform boundedness of partial-sum operators on bounded Vilenkin systems}\label{sec:proof of bounded}
Here we provide a detailed argument for the case of bounded Vilenkin systems to illustrate the novelties of the present paper.

By the Paley conjugation identity---Lemma \ref{lem:paley-conjugation-identity}, one rewrites $S_nf$ as
\begin{align*}
    S_nf&{=}\psi_n\sum_{0\le k\le K}\sum_{\ell=m_k-n_k}^{m_k-1}\Delta_{k,\ell}(\overline{\psi_n}f)=\psi_n\sum_{\ell=1}^{\sup_km_k-1}\sum_{\substack{0\le k\le K\\ m_k-n_k\le \ell\le m_k-1}}\Delta_{k,\ell}(\overline{\psi_n}f).
\end{align*}
By the triangle inequality, one may put the sum over $\ell$ outside of the norm,
\begin{align*}
 \|S_nf\|_{L^p(\Gm;X)}\le \sum_{\ell=1}^{\sup_km_k-1}\left\|\sum_{\substack{0\le k\le K\\ m_k-n_k\le \ell\le m_k-1}}\Delta_{k,\ell}(\overline{\psi_n}f)\right\|_{L^p(\Gm;X)}.
 \end{align*}
Now note that for a fixed $\ell\ge1$ and each $k$ with $\ell<m_k$, $\Delta_{k,\ell}(g)$ is $\mathcal F_{k+1}$-measurable and has conditional expectation zero onto $\mathcal F_k$. Thus, one may use the UMD property and the complex contraction principle (see \cite[Proposition~3.2.10]{HNVWI}) to remove the restriction $m_k-n_k\le\ell$ while retaining $\ell<m_k$,
 \begin{align*}
 \|S_nf\|_{L^p(\Gm;X)}&\lesssim \beta_{p,X}\sum_{\ell=1}^{\sup_km_k-1}\left\|\sum_{\substack{0\le k\le K\\ m_k-n_k\le \ell\le m_k-1}}\varepsilon_k\Delta_{k,\ell}(\overline{\psi_n}f)\right\|_{L^p(\Gm\times \Omega_\varepsilon;X)}\\
 &\lesssim \beta_{p,X}\sum_{\ell=1}^{\sup_km_k-1}\left\|\sum_{\substack{0\le k\le K\\ \ell<m_k}}\varepsilon_k\Delta_{k,\ell}(\overline{\psi_n}f)\right\|_{L^p(\Gm\times \Omega_\varepsilon;X)}.
 \end{align*}
We now use the identity \eqref{finer} and apply the vector-valued Stein inequality (see e.g. \cite[Theorem~4.2.23]{HNVWI}), the complex contraction principle, and the UMD property to conclude that
\begin{align*}
  \|S_nf\|_{L^p(\Gm;X)}
 &\lesssim \beta_{p,X}\sum_{\ell=1}^{\sup_km_k-1}\left\|\sum_{\substack{0\le k\le K\\ \ell<m_k}}\varepsilon_k
\psi_{\ell M_k}E_k(\overline{\psi_{\ell M_k}}d_k(\overline{\psi_n}f))\right\|_{L^p(\Gm\times \Omega_\varepsilon;X)}\\
 &\lesssim \beta_{p,X}\sum_{\ell=1}^{\sup_km_k-1}\left\|\sum_{\substack{0\le k\le K\\ \ell<m_k}}\varepsilon_k
E_k(\overline{\psi_{\ell M_k}}d_k(\overline{\psi_n}f))\right\|_{L^p(\Gm\times \Omega_\varepsilon;X)}\\
 &\lesssim \beta_{p,X}^2\sum_{\ell=1}^{\sup_km_k-1}\left\|\sum_{\substack{0\le k\le K\\ \ell<m_k}}\varepsilon_k
\overline{\psi_{\ell M_k}}d_k(\overline{\psi_n}f)\right\|_{L^p(\Gm\times \Omega_\varepsilon;X)}\\
 &\lesssim (\sup_k m_k) \beta_{p,X}^3 \|\overline{\psi_n}f\|_{L^p(\Gm;X)}\lesssim (\sup_k m_k) \beta_{p,X}^3 \|f\|_{L^p(\Gm;X)}.
\end{align*}

\section{From cyclic projections decoupling property to Vilenkin partial sums}
\label{sec:CPD-to-partial-sums}

In this section, we first introduce the definition of the {\it cyclic projections decoupling property}, and state the key characterization of the UMD property---Theorem \ref{thm:umd-implies-CPD}. We then recall McConnell's decoupling inequality, and combine it with the CPD property to prove the desired inequality for terminal projections $(\pi_{k,n_k})_k$. Finally, we invoke the Paley conjugation identity from Section~\ref{s2.2}, which expresses a
partial-sum operator in terms of terminal projections acting on martingale differences, to complete the proof of Theorem~\ref{thm:uniform-umd-vilenkin-partial-sums}.

\subsection{Cyclic projections decoupling property}
In this subsection, we introduce a geometric property of Banach spaces,
which we call the \emph{cyclic projections decoupling property}, abbreviated
as the \(\operatorname{CPD}\) property.
For an \(X\)-valued function
\(h:\mathbb Z_m\to X\), we define its Fourier transform by
\[
\widehat h(q)
:=
\frac1m\sum_{j=0}^{m-1}h(j)e^{-2\pi iqj/m},
\qquad q\in\mathbb Z_m.
\]
For \(0\le a<m\), the terminal cyclic Fourier projection is defined by
\begin{equation}\label{def of Q}
Q_{m,a}h(j)
:=
\sum_{q=m-a}^{m-1}
\widehat h(q)e^{2\pi iqj/m},
\qquad
Q_{m,0}:=0.
\end{equation}
We also write
\[
L_0^p(\mathbb Z_m;X)
:=
\left\{
h\in L^p(\mathbb Z_m;X):
\frac1m\sum_{j=0}^{m-1}h(j)=0
\right\}.
\]

The cyclic projections decoupling property requires these terminal
projections to act uniformly on sums of independent mean-zero
functions, as follows.

\begin{definition}[Cyclic projections decoupling property]\label{def:CPD-property}
Let \(1<p<\infty\). A Banach space \(X\) is said to have the
\(\operatorname{CPD}_p\) property if there exists a constant
\(C<\infty\) such that, for every \(N\ge1\), every choice
of integers \(m_k\ge2\), every \(0\le a_k<m_k\) and
\(
h_k\in L_0^p(\mathbb Z_{m_k};X),1\le k\le N,
\)
one has
\begin{equation}\label{eq:CPD-property}
\left\|
\sum_{k=1}^N Q_{m_k,a_k}h_k(y_k)
\right\|_{L^p\left(\prod_{k=1}^N\mathbb Z_{m_k};X\right)}
\le
C
\left\|
\sum_{k=1}^N h_k(y_k)
\right\|_{L^p\left(\prod_{k=1}^N\mathbb Z_{m_k};X\right)}.
\end{equation}
The infimum of $C$ is denoted by $\gamma_{p,X}.$
\end{definition}

To the best of our knowledge, the uniform Fourier-analytic property
formulated above has not previously been studied in this form.
Nevertheless, the underlying probabilistic structure is closely related to
the classical theories of Hoeffding decompositions, independent
\(L^p\)-sums, decoupling, and, in particular, the first Hoeffding chaos. These topics have been extensively studied; we refer the reader to \cite{Hoeffding1948,BourgainRosenthalSchechtman1981,deLaPenaMontgomerySmith1995,Kwapien1987,Kwapien2010,CoxGeiss2021,RzeszutWojciechowski2021} for further background and related results.

The
following theorem shows that the $\operatorname{CPD}_p$ property characterizes UMD Banach spaces.
\begin{theorem}\label{thm:umd-implies-CPD}
Let \(X\) be a Banach space and let \(1<p<\infty\). Then \(X\) has the
\(\operatorname{CPD}_p\) property if and only if \(X\) is \(\operatorname{UMD}\). Moreover, we have $\beta^{1/2}_{p,X}\lesssim_p\gamma_{p,X}\lesssim_p \beta_{p,X}^2.$
\end{theorem}

As a consequence of Proposition \ref{prop:UMD} and Theorem \ref{thm:umd-implies-CPD}, the definition of $\operatorname{CPD}_p$ property is also independent of the exponent $p.$
\begin{cor}
     The \(\operatorname{CPD}_p\) property does not depend on the choice of exponent $p.$ Equivalently, if \eqref{eq:CPD-property} holds for one
\(p\in(1,\infty)\), then it holds for every \(p\in(1,\infty)\).
\end{cor}

\begin{remark}
    Let \(\Omega:=\prod_{k=1}^N\mathbb Z_{m_k}\), and define the first Hoeffding projection by
\[
\mathsf P_1f:=\sum_{k=1}^N\bigl(\mathbb E[f\,|\, y_k]-\mathbb Ef\bigr).
\]
On the first-chaos subspace, define
\[
\mathcal Q_{\boldsymbol m,\boldsymbol a}
\left(\sum_{k=1}^Nh_k(y_k)\right)
:=
\sum_{k=1}^NQ_{m_k,a_k}h_k(y_k),
\qquad \mathbb Eh_k=0.
\]
As observed by Rzeszut and Wojciechowski~\cite{RzeszutWojciechowski2021}, \(\mathsf P_1\) is uniform bounded on
\(L^p(\Omega;X)\) with respect to $N$ for any \(\operatorname{UMD}\) Banach space. Hence Theorem~\ref{thm:umd-implies-CPD} implies that
\(\mathcal Q_{\boldsymbol m,\boldsymbol a}\circ\mathsf P_1\)
is uniform bounded on \(L^p(\Omega;X)\) with respect to $N$ and $\mathbf m$.
\end{remark}

The proof of Theorem~\ref{thm:umd-implies-CPD} is carried out in
Sections~\ref{sec:decoupling-estimates} and~\ref{sec:proof-CPD}.
We first establish two decoupling estimates in
Section~\ref{sec:decoupling-estimates}:
a diagonal inequality for finitely supported Marcinkiewicz multipliers
on the torus and an inequality for cyclic convolution operators.
We then combine these estimates through an exact sampling
factorization to obtain the centered cyclic projection inequality in Section~\ref{sec:proof-CPD}.
From this, we derive the terminal projection bound
\eqref{eq:CPD-property}; the converse implication in the characterization
then follows by an approximation argument involving the Riesz projection. To the best of our knowledge, these Fourier-analytic decoupling
inequalities are new in the precise forms established here.

\subsection{Adapted martingale differences and decoupling inequality}

We shall use McConnell's decoupling theorem for tangent
martingale-difference sequences in UMD spaces; see
\cite{McConnell1989}. 
We first recall
the definition of a decoupled tangent sequence.

\begin{definition}[Decoupled tangent sequences]
Let \(X\) be a separable Banach space, and let
\[
(\Omega,\mathcal F,\mathbb P,(\mathcal F_k)_{k=0}^N)
\]
be a filtered probability space. Let \((d_k)_{k=1}^N\) and
\((e_k)_{k=1}^N\) be \(\mathcal F_k\)-adapted sequences of
\(X\)-valued random variables. The sequence \((e_k)_{k=1}^N\) is called a
\emph{decoupled tangent sequence} of \((d_k)_{k=1}^N\) if there exists a
sub-\(\sigma\)-algebra \(\mathcal G\subseteq\mathcal F\) such that
the following conditions hold:

\begin{enumerate}
\item[(i)] \emph{Tangency.} For every \(k=1,\ldots,N\) and every
\(B\in\mathcal B(X)\),
\[
\mathbb P(d_k\in B\mid\mathcal F_{k-1})
=
\mathbb P(e_k\in B\mid\mathcal F_{k-1})
=
\mathbb P(e_k\in B\mid\mathcal G)
\qquad\text{a.s.}
\]

\item[(ii)] \emph{Conditional independence.} For every
\(B_1,\ldots,B_N\in\mathcal B(X)\),
\[
\mathbb P
\bigl(
e_1\in B_1,\ldots,e_N\in B_N
\mid\mathcal G
\bigr)
=
\prod_{k=1}^N
\mathbb P(e_k\in B_k\mid\mathcal G)
\qquad\text{a.s.}
\]
\end{enumerate}
\end{definition}

In the form needed below, McConnell's decoupling inequality states that
if \((d_k)_k\) is a finite \(X\)-valued martingale-difference sequence
and \((e_k)_k\) is a decoupled tangent sequence of \((d_k)_k\), then
\begin{equation}\label{eq:umd-tangent-decoupling}
\beta_{p,X}^{-1}
\left\|\sum_ke_k\right\|_{L^p}\le \left\|\sum_kd_k\right\|_{L^p}
\le \beta_{p,X}
\left\|\sum_ke_k\right\|_{L^p}.
\end{equation}

For \(0\le a<m_k\), recall that we have defined
\[
\pi_{k,a}
:=
\sum_{\ell=m_k-a}^{m_k-1}\Delta_{k,\ell},
\qquad
\pi_{k,0}:=0.
\]
Let \(d_k:=E_{k+1}-E_k\) be the martingale-difference operator with respect to the canonical filtration of $\Gm$ defined in Section \ref{s2.1}. Let \(Q_{m_k,a}^{(x_k)}\) denote the terminal cyclic projection
\(Q_{m_k,a}\) acting in the \(x_k\)-coordinate. 
Then we have
\begin{align}\label{coincide}
    \pi_{k,a}d_k=Q_{m_k,a}^{(x_k)}d_k.
\end{align}
Indeed, by linearity,
it suffices to verify the identity on Vilenkin characters in the range
of \(d_k\). Every such character satisfies \(n_k\ne0\), \(n_j=0\) for
\(j>k\), and can be written as
\[
\psi_n(x)
=
\prod_{j=0}^k r_j(x)^{n_j}.
\]
Consequently,
\begin{align*}
\pi_{k,a}\psi_n(x)
&=
\prod_{j=0}^{k-1}r_j(x)^{n_j}
e^{2\pi in_kx_k/m_k}
\mathbf1_{\{m_k-a\le n_k<m_k\}}
\\
&=
\prod_{j=0}^{k-1}r_j(x)^{n_j}
Q_{m_k,a}
\left(
e^{2\pi in_k\,\cdot/m_k}
\right)(x_k)
=
Q_{m_k,a}^{(x_k)}\psi_n(x).
\end{align*}

This coordinate identity allows us to apply the independent
first-chaos estimate after decoupling the martingale differences.
The resulting bound is the main step in the reduction.

\begin{proposition}
\label{prop:adapted-terminal-core}
Let \(1<p<\infty\). Let \(X\) be a \(\operatorname{UMD}\) Banach space. For every generating sequence
\(\mathbf m\), every \(K\ge0\), every sequence
\((a_k)_{k=0}^K\) satisfying \(0\le a_k<m_k\), and every
\(f\in L^p(\Gm;X)\),
\begin{equation}\label{eq:adapted-terminal-core}
\left\|
\sum_{k=0}^K\pi_{k,a_k}d_kf
\right\|_{L^p(\Gm;X)}
\le \gamma_{p,X}\beta_{p,X}^2
\left\|
\sum_{k=0}^Kd_kf
\right\|_{L^p(\Gm;X)}.
\end{equation}
\end{proposition}

\begin{proof}
Since \(f\) is strongly measurable, it is essentially valued in a
separable closed subspace \(X_f\subseteq X\). All conditional
expectations and scalar Fourier projections appearing below preserve
\(X_f\), and every closed subspace of a UMD space is UMD. Replacing \(X\)
by \(X_f\), we may therefore assume that \(X\) is separable.

\noindent
\textit{Step 1: Enlargement of the probability space.}
We follow the construction for an enlargement of the probability space; see \cite[Theorem~4.4.11 and Example~4.4.13]{HNVWI}.
Let
\(
Y_K:=\prod_{k=0}^K\mathbb Z_{m_k},
\)
and denote its normalized Haar probability measure by \(\nu_K\). Consider the
product probability space
\[
(\widehat\Omega,\widehat{\mathcal F},\widehat{\mathbb P})
:=
\left(
\Gm\times Y_K,\,
\mathcal B(\Gm)\otimes\mathcal B(Y_K),\,
\mu\otimes\nu_K
\right).
\]
We write \((x,y)\) for the coordinates on \(\widehat\Omega\), where
\[
x=(x_j)_{j\ge0},
\qquad
y=(y_0,\ldots,y_K).
\]
Every random variable on \(\Gm\) is identified with its canonical
extension to \(\widehat\Omega\), constant in the \(y\)-variables.
Define a filtration on \(\widehat\Omega\) by
\[
\widehat{\mathcal F}_0
:=
\{\varnothing,\widehat\Omega\},
\qquad
\widehat{\mathcal F}_{k+1}
:=
\sigma(x_0,\ldots,x_k,y_0,\ldots,y_k),
\qquad
0\le k\le K.
\]
Let \(\widehat E_k\) denote conditional expectation with respect to
\(\widehat{\mathcal F}_k\).
We may rewrite
\[
d_kf(x)
=
h_k(x_0,\ldots,x_{k-1},x_k).
\]
For every fixed
\[
x_{<k}:=(x_0,\ldots,x_{k-1}),
\]
the mean-zero property of \(d_kf\) gives
\[
\widehat E_k(d_kf)=\frac1{m_k}
\sum_{s\in\mathbb Z_{m_k}}
h_k(x_{<k},s)
=
0.
\]
Consequently,
\((d_kf)_{k=0}^K=(h_k)_{k=0}^K\) is also a martingale-difference
sequence with respect to the new filtration
\((\widehat{\mathcal F}_k)_{k=0}^{K+1}\).

\noindent
\textit{Step 2: Tangency of the unprojected sequence.}

On the enlarged probability space, define
\[
e_k(x,y)
:=
h_k(x_{<k},y_k),
\qquad
0\le k\le K.
\]
To match the indexing in the preceding definition, set
\[
D_{k+1}:=d_kf,
\qquad
D_{k+1}':=e_k,
\qquad
0\le k\le K.
\]
Both \(D_{k+1}\) and \(D_{k+1}'\) are
\(\widehat{\mathcal F}_{k+1}\)-measurable, so both sequences are adapted
to \((\widehat{\mathcal F}_j)_{j=0}^{K+1}\).
Let
\[
\mathcal G
:=
\sigma(x_0,x_1,\ldots)
=
\mathcal B(\Gm)\otimes\{\varnothing,Y_K\}
\subseteq\widehat{\mathcal F}.
\]
Let \(\mathcal B(X)\) denote the Borel \(\sigma\)-algebra of \(X\).
For every \(B\in\mathcal B(X)\), the independence and uniform
distribution of \(x_k\) and \(y_k\) give
\begin{align*}
\widehat{\mathbb P}\bigl(
D_{k+1}\in B\mid\widehat{\mathcal F}_k
\bigr)
&=
\frac1{m_k}
\sum_{s\in\mathbb Z_{m_k}}
\mathbf1_B\bigl(h_k(x_{<k},s)\bigr)
\nonumber\\
&=
\widehat{\mathbb P}\bigl(
D_{k+1}'\in B\mid\widehat{\mathcal F}_k
\bigr)
=
\widehat{\mathbb P}\bigl(
D_{k+1}'\in B\mid\mathcal G
\bigr)
\qquad\text{a.s.}
\end{align*}
Moreover, for arbitrary
\(B_1,\ldots,B_{K+1}\in\mathcal B(X)\), the independence of
\(y_0,\ldots,y_K\) implies
\begin{align*}
\widehat{\mathbb P}\bigl(
D_1'\in B_1,\ldots,D_{K+1}'\in B_{K+1}
\mid\mathcal G
\bigr)
&=
\prod_{k=0}^K
\left[
\frac1{m_k}
\sum_{s\in\mathbb Z_{m_k}}
\mathbf1_{B_{k+1}}\bigl(h_k(x_{<k},s)\bigr)
\right]
\\
&\qquad=
\prod_{k=0}^K
\widehat{\mathbb P}\bigl(
D_{k+1}'\in B_{k+1}\mid\mathcal G
\bigr)
\qquad\text{a.s.}
\end{align*}
Thus \((D_{k+1}')_{k=0}^K\), equivalently \((e_k)_{k=0}^K\), is a
decoupled tangent sequence of \((D_{k+1})_{k=0}^K\), equivalently
\((d_kf)_{k=0}^K\).

\noindent
\textit{Step 3: Tangency of the projected sequence.}

Define
\[
\widetilde d_k
:=
Q_{m_k,a_k}^{(x_k)}d_kf,
\qquad
\widetilde e_k
:=
Q_{m_k,a_k}^{(y_k)}e_k,\qquad 0\le k\le K.
\]
For every fixed \(x_{<k}\), set
\[
H_{k,x_{<k}}
:=
Q_{m_k,a_k}
\bigl(h_k(x_{<k},\cdot)\bigr).
\]
Then we have
\[
\widetilde d_k(x)
=
H_{k,x_{<k}}(x_k),
\qquad
\widetilde e_k(x,y)
=
H_{k,x_{<k}}(y_k).
\]
To match the indexing in the preceding definition, set
\[
\widetilde D_{k+1}:=\widetilde d_k,
\qquad
\widetilde D_{k+1}':=\widetilde e_k,
\qquad
0\le k\le K.
\]
Both random variables are
\(\widehat{\mathcal F}_{k+1}\)-measurable.

As in Step~2, for every \(B\in\mathcal B(X)\),
\begin{align*}
\widehat{\mathbb P}\bigl(
\widetilde D_{k+1}\in B
\mid\widehat{\mathcal F}_k
\bigr)
&=
\frac1{m_k}
\sum_{s\in\mathbb Z_{m_k}}
\mathbf1_B\bigl(H_{k,x_{<k}}(s)\bigr)
\nonumber\\
&=
\widehat{\mathbb P}\bigl(
\widetilde D_{k+1}'\in B
\mid\widehat{\mathcal F}_k
\bigr)
=
\widehat{\mathbb P}\bigl(
\widetilde D_{k+1}'\in B
\mid\mathcal G
\bigr)
\qquad\text{a.s.}
\end{align*}
Furthermore, for arbitrary
\(B_1,\ldots,B_{K+1}\in\mathcal B(X)\),
\begin{align*}
\widehat{\mathbb P}\bigl(
\widetilde D_1'\in B_1,\ldots,
\widetilde D_{K+1}'\in B_{K+1}
\mid\mathcal G
\bigr)
&=
\prod_{k=0}^K
\left[
\frac1{m_k}
\sum_{s\in\mathbb Z_{m_k}}
\mathbf1_{B_{k+1}}\bigl(H_{k,x_{<k}}(s)\bigr)
\right]
\\
&=
\prod_{k=0}^K
\widehat{\mathbb P}\bigl(
\widetilde D_{k+1}'\in B_{k+1}
\mid\mathcal G
\bigr)
\qquad\text{a.s.}
\end{align*}
Therefore,
\((\widetilde D_{k+1}')_{k=0}^K\), equivalently
\((\widetilde e_k)_{k=0}^K\), is a decoupled tangent sequence of
\((\widetilde D_{k+1})_{k=0}^K\), equivalently
\((\widetilde d_k)_{k=0}^K\).
It remains to verify that
\((\widetilde D_{k+1})_{k=0}^K\) is a martingale-difference sequence.
For every fixed \(x_{<k}\), the terminal frequency set
\([m_k-a_k,m_k)\subseteq\mathbb Z_{m_k}\) does not contain the zero
frequency. Hence
\[
\widehat E_k\widetilde D_{k+1}
=
\frac1{m_k}\sum_{s\in\mathbb Z_{m_k}}H_{k,x_{<k}}(s)
=0.
\]
Thus \((\widetilde D_{k+1})_{k=0}^K\) is a martingale-difference sequence
with respect to \((\widehat{\mathcal F}_j)_{j=0}^{K+1}\).

\noindent
\textit{Step 4: Application of the decoupling inequalities.}

For brevity, write
\[
L_x^p:=L^p(\Gm;X),
\qquad
L_{x,y}^p:=L^p(\widehat\Omega;X).
\]
Applying \eqref{coincide}, the decoupling inequality \eqref{eq:umd-tangent-decoupling} and
Theorem~\ref{thm:umd-implies-CPD}, we obtain
\begin{eqnarray*}
\left\|
\sum_{k=0}^K \pi_{k,a_k}d_kf
\right\|_{L_x^p}
&\overset{\eqref{coincide}}{=}&
\left\|
\sum_{k=0}^K \widetilde d_k
\right\|_{L_x^p}
=
\left\|
\sum_{k=0}^K \widetilde d_k
\right\|_{L_{x,y}^p}
\\
&\overset{\eqref{eq:umd-tangent-decoupling}}{\le} & \beta_{p,X}
\left\|
\sum_{k=0}^K \widetilde e_k
\right\|_{L_{x,y}^p}
\overset{\eqref{eq:CPD-property}}{\le} \gamma_{p,X}\beta_{p,X}
\left\|
\sum_{k=0}^K e_k
\right\|_{L_{x,y}^p}
\\
&\overset{\eqref{eq:umd-tangent-decoupling}}{\le} & \gamma_{p,X}\beta^2_{p,X}
\left\|
\sum_{k=0}^K d_kf
\right\|_{L_{x,y}^p}
=
\gamma_{p,X}\beta^2_{p,X}\left\|
\sum_{k=0}^K d_kf
\right\|_{L_x^p}.
\end{eqnarray*}
This proves \eqref{eq:adapted-terminal-core}.
\end{proof}

\subsection{Proof of the main theorem}

We now complete the reduction of
Theorem~\ref{thm:uniform-umd-vilenkin-partial-sums} to
Theorem~\ref{thm:umd-implies-CPD}.
\begin{proof}[Proof of
Theorem~\ref{thm:uniform-umd-vilenkin-partial-sums}]
Let \(K\) be the largest number $k$ for which \(n_k\ne0\), and set
\(
g:=\overline{\psi_n}f.
\)
By Lemma~\ref{lem:paley-conjugation-identity} and
Proposition~\ref{prop:adapted-terminal-core},
\begin{align*}
\left\|
S_nf
\right\|_{L^p}
&=
\left\|
P_{B_n} g
\right\|_{L^p}
=
\left\|
\sum_{k=0}^K
\pi_{k,n_k}d_kg
\right\|_{L^p}
\\
&\le \gamma_{p,X}\beta_{p,X}^2
\left\|
\sum_{k=0}^Kd_kg
\right\|_{L^p}
=
\gamma_{p,X}\beta_{p,X}^2\|E_{K+1}g-E_0g\|_{L^p}\\
&
\le
2\gamma_{p,X}\beta_{p,X}^2\|g\|_{L^p}
=
2\gamma_{p,X}\beta_{p,X}^2\|f\|_{L^p}.
\end{align*}
Since $\gamma_{p,X}\lesssim_p\beta_{p,X}^2,$ we conclude that $\|S_nf\|_{L^p}\lesssim_p\beta_{p,X}^4\|f\|_{L^p}$.
\end{proof}

\section{Two decoupling estimates on the tori and cyclic groups}
\label{sec:decoupling-estimates}

We establish the decoupling inequalities for Marcinkiewicz multipliers on the tori and convolution operators on the cyclic groups, which will be used
to prove Theorem~\ref{thm:umd-implies-CPD}.

\subsection{Decoupling estimate for Marcinkiewicz multipliers on the tori}
\label{sec:torus-decoupling}

For a finitely supported scalar sequence
\(b:\mathbb Z\to\mathbb C\), define
\[
\|b\|_{\operatorname{var}}
:=
\sum_{n\in\mathbb Z}|b(n)-b(n-1)|
\]
and the relevant Marcinkiewicz multiplier
\[
T_b^{\mathbb T}f(t)
:=
\sum_{n\in\mathbb Z}b(n)\widehat f(n)e^{2\pi int}.
\]
The first decoupling inequality controls these multipliers uniformly
by the variation of their symbols. 

\begin{theorem}[A decoupling  estimate for Marcinkiewicz multipliers]
\label{thm:diagonal-bv-independent}
Let \(1<p<\infty\), and let \(X\) be a \(\operatorname{UMD}\) Banach space. For
\(1\le k\le N\), let \(f_k\in L_0^p(\mathbb T;X)\) and let
\(b_k:\mathbb Z\to\mathbb C\) be finitely supported.
If
\[
\sup_{1\le k\le N}\|b_k\|_{\mathrm{var}}\le V,
\]
then
\begin{equation}\label{eq:diagonal-bv-independent}
\left\|
\sum_{k=1}^N
T_{b_k}^{\mathbb T}f_k(t_k)
\right\|_{L^p(\mathbb T^N;X)}
\lesssim_p \beta_{p,X}^2
V
\left\|
\sum_{k=1}^Nf_k(t_k)
\right\|_{L^p(\mathbb T^N;X)}.
\end{equation}
\end{theorem}

We record two standard conclusions that will be
used repeatedly. 
\begin{lemma}\label{lem:independent-umd-randomization}
Let \(1<p<\infty\), let \(X\) be a Banach space,
and let \((\Omega_k,\mathbb P_k)_{k=1}^N\) be probability spaces. Suppose
that \(f_k\in L_0^p(\Omega_k;X)\). Regarding \(f_k\)'s as independent
random variables on \(\prod_{k=1}^N\Omega_k\), we have
\begin{equation}\label{eq:independent-umd-randomization}
\left\|
\sum_{k=1}^N f_k
\right\|_{L^p(\prod_k\Omega_k;X)}
\simeq
\left\|
\sum_{k=1}^N\varepsilon_k f_k
\right\|_{L^p(\Omega_\varepsilon\times\prod_k\Omega_k;X)}.
\end{equation}
Moreover, for every scalar sequence \((\lambda_k)_{k=1}^N\) satisfying
\(|\lambda_k|\le1\),
\begin{equation}\label{eq:independent-complex-contraction}
\left\|
\sum_{k=1}^N\lambda_k f_k
\right\|_{L^p(\prod_k\Omega_k;X)}
\lesssim
\left\|
\sum_{k=1}^N f_k
\right\|_{L^p(\prod_k\Omega_k;X)}.
\end{equation}
\end{lemma}
\begin{proof}
    Let \(f_1',\ldots,f_N'\) be an independent copy of
\(f_1,\ldots,f_N\). For simplicity, let $\|\cdot\|_{L_{f,f'}^p}$ denote the $L^p$ norm taken over $f_1,\dots,f_N$ and $f_1',\dots,f_N'.$ Then
\[
\begin{aligned}
\left\|\sum_{k=1}^N f_k\right\|_{L^p}
&=
\left\|
\mathbb E_{f'}
\sum_{k=1}^N (f_k-f_k')
\right\|_{L^p}
\\
&\le
\mathbb E_{f'}
\left\|
\sum_{k=1}^N (f_k-f_k')
\right\|_{L^p}
\\
&\le
\left\|
\sum_{k=1}^N (f_k-f_k')
\right\|_{L^p_{f,f'}}.
\end{aligned}
\]
Since \(f_k-f_k'\), \(1\le k\le N\), are independent and symmetric,
we have
\[
\varepsilon_k(f_k-f_k')
\overset{d}{=}
f_k-f_k'.
\]
Hence
\[
\begin{aligned}
\left\|\sum_{k=1}^N f_k\right\|_{L^p}
&\le
 \left(\mathbb E_{\varepsilon}\left\|
\sum_{k=1}^N
\varepsilon_k(f_k-f_k')
\right\|_{L^p_{f,f'}}^p\right)^{1/p}
\\
&\le
\left(\mathbb E_{\varepsilon}\left\|
\sum_{k=1}^N \varepsilon_k f_k
\right\|_{L^p}^p\right)^{1/p}
+
\left(\mathbb E_{\varepsilon}\left\|
\sum_{k=1}^N \varepsilon_k f_k'
\right\|_{L^p}^p\right)^{1/p}
\\
&=
2\left(\mathbb E_{\varepsilon}
\left\|
\sum_{k=1}^N \varepsilon_k f_k
\right\|_{L^p}^p\right)^{1/p}.
\end{aligned}
\]
Similarly, we have the converse inequality 
$$\left(\mathbb E_{\varepsilon}
\left\|
\sum_{k=1}^N \varepsilon_k f_k
\right\|_{L^p}^p\right)^{1/p}\le 2\left\|\sum_{k=1}^N f_k\right\|_{L^p}.$$
Now \eqref{eq:independent-complex-contraction} follows from \eqref{eq:independent-umd-randomization} and the complex contraction principle.
\end{proof}

For an interval \(I\subset\mathbb Z\), let \(P_I^{\mathbb T}\) denote the
corresponding Fourier projection. The next lemma is an easy consequence of Lemma \ref{Bochner1} and Bourgain's celebrated Riesz-projection characterization of UMD spaces---Remark \ref{Riesz}. The resulting statement is standard and should be well known to experts; see e.g. \cite[Lemma~1.10]{ArendtBu2002}. We omit the proof.

\begin{lemma}\label{lem:randomized-torus-intervals}
Let \(1<p<\infty\), and let \(X\) be a \(\operatorname{UMD}\) Banach
space. For arbitrary finite intervals
\(I_1,\ldots,I_N\subset\mathbb Z\) and arbitrary
\(f_1,\ldots,f_N\in L^p(\mathbb T;X)\),
\[
\left\|
\sum_{k=1}^N\varepsilon_kP_{I_k}^{\mathbb T}f_k
\right\|_{L^p(\Omega_\varepsilon\times\mathbb T;X)}
\lesssim_p \beta_{p,X}^2
\left\|
\sum_{k=1}^N\varepsilon_kf_k
\right\|_{L^p(\Omega_\varepsilon\times\mathbb T;X)}.
\]
\end{lemma}


We are in a position to prove the decoupling Marcinkiewicz multiplier estimate on the tori.

\begin{proof}[Proof of Theorem~\ref{thm:diagonal-bv-independent}]
Assume that
\[
\operatorname{supp}b_k\subset [u_k,v_k]\cap\mathbb Z,
\]
and define
\[
c_{k,\ell}:=b_k(\ell)-b_k(\ell-1),
\qquad u_k\le\ell\le v_k.
\]
Since \(b_k(u_k-1)=0\), telescoping yields
\begin{equation}\label{eq:bv-interval-decomposition}
b_k
=
\sum_{\ell=u_k}^{v_k}
c_{k,\ell}\mathbf1_{[\ell,v_k]}.
\end{equation}
Furthermore,
\[
C_k
:=
\sum_{\ell=u_k}^{v_k}|c_{k,\ell}|
\le
\sum_{n\in\mathbb Z}|b_k(n)-b_k(n-1)|
=
\|b_k\|_{\mathrm{var}}
\le V.
\]
In particular, \(\|b_k\|_{\ell^\infty(\mathbb Z)}\le V\). If \(V=0\),
then every \(b_k\) is constant and finitely supported, hence \(b_k=0\).
The conclusion is therefore immediate. We may assume that \(V>0\).

For each \(k\), introduce an auxiliary random variable
\[
\nu_k\in\{u_k,u_{k}+1,\ldots,v_k\}\cup\{*\}
\]
with distribution
\[
\mathbb P(\nu_k=\ell)
=
\frac{|c_{k,\ell}|}{V},
\qquad u_k\le\ell\le v_k,
\]
and
\[
\mathbb P(\nu_k=*)
=
1-\frac{C_k}{V}.
\]
The probability at the dummy value \(*\) makes this a probability distribution.
Set
\[
I_{k,\ell}:=[\ell,v_k]\cap\mathbb Z
\]
and
\[
\theta_{k,\ell}
:=
\begin{cases}
\dfrac{c_{k,\ell}}{|c_{k,\ell}|},
& c_{k,\ell}\neq0,\\[6pt]
0,
& c_{k,\ell}=0.
\end{cases}
\]
For the dummy value, let \(I_{k,*}:=\{0\}\) and
\(\theta_{k,*}:=0\). Then \(|\theta_{k,\nu_k}|\le1\), and
\eqref{eq:bv-interval-decomposition} gives
\begin{align*}
V\,\mathbb E_{\nu_k}
\bigl[
\theta_{k,\nu_k}P_{I_{k,\nu_k}}^{\mathbb T}
\bigr]
&=
V\sum_{\ell=u_k}^{v_k}
\frac{|c_{k,\ell}|}{V}
\theta_{k,\ell}P_{I_{k,\ell}}^{\mathbb T}=
\sum_{\ell=u_k}^{v_k}
c_{k,\ell}P_{[\ell,v_k]}^{\mathbb T}
=
T_{b_k}^{\mathbb T}.
\end{align*}
Choose the variables \((\nu_k)_{k=1}^N\) independently of one another and
independently of the variables \(t_k\) and \(\varepsilon_k\). Write
\(\nu=(\nu_1,\ldots,\nu_N)\), and let \(\mathbb E_\nu\) denote expectation
with respect to \(\nu\). By linearity,
\begin{equation}\label{eq:bv-random-representation}
\sum_{k=1}^N
\varepsilon_kT_{b_k}^{\mathbb T}f_k(t_k)
=
V\,\mathbb E_{\nu}
\left[
\sum_{k=1}^N
\varepsilon_k\theta_{k,\nu_k}
P_{I_{k,\nu_k}}^{\mathbb T}f_k(t_k)
\right].
\end{equation}
Since \(f_k\in L_0^p(\mathbb T;X)\),
\[
\int_{\mathbb T}T_{b_k}^{\mathbb T}f_k\,dt
=
b_k(0)\widehat f_k(0)
=
0.
\]
Thus the functions
\(T_{b_k}^{\mathbb T}f_k(t_k)\), \(1\le k\le N\), are independent and
mean-zero. Lemma~\ref{lem:independent-umd-randomization}, followed by
\eqref{eq:bv-random-representation}, Jensen's inequality, and the complex
contraction principle, yields
\begin{align*}
\left\|
\sum_{k=1}^N
T_{b_k}^{\mathbb T}f_k(t_k)
\right\|_{L^p(\mathbb T^N;X)}
&\le 2
\left\|
\sum_{k=1}^N
\varepsilon_kT_{b_k}^{\mathbb T}f_k(t_k)
\right\|_{L^p(\Omega_\varepsilon\times\mathbb T^N;X)}
\\
&\le
2V\,\mathbb E_\nu
\left\|
\sum_{k=1}^N
\varepsilon_k\theta_{k,\nu_k}
P_{I_{k,\nu_k}}^{\mathbb T}f_k(t_k)
\right\|_{L^p(\Omega_\varepsilon\times\mathbb T^N;X)}
\\
&\le 8
V\,\mathbb E_\nu
\left\|
\sum_{k=1}^N
\varepsilon_k
P_{I_{k,\nu_k}}^{\mathbb T}f_k(t_k)
\right\|_{L^p(\Omega_\varepsilon\times\mathbb T^N;X)}.
\end{align*}
It consequently suffices to prove, uniformly over every fixed choice of
\((\nu_k)_{k=1}^N\), that
\begin{equation}\label{eq:fixed-random-interval-choice}
\left\|
\sum_{k=1}^N
\varepsilon_k
P_{I_{k,\nu_k}}^{\mathbb T}f_k(t_k)
\right\|_{L^p(\Omega_\varepsilon\times\mathbb T^N;X)}
\lesssim_p \beta_{p,X}^2
\left\|
\sum_{k=1}^N
\varepsilon_kf_k(t_k)
\right\|_{L^p(\Omega_\varepsilon\times\mathbb T^N;X)}.
\end{equation}

To prove \eqref{eq:fixed-random-interval-choice}, introduce an additional
variable \(u\in\mathbb T\) and define
\[
F_k(u,t_1,\ldots,t_N):=f_k(t_k+u).
\]
Indeed, if
\[
f_k(t_k)=\sum_{n\in\mathbb Z}\widehat f_k(n)e^{2\pi int_k},
\]
then
\[
F_k(u,t_1,\ldots,t_N)
=
\sum_{n\in\mathbb Z}
\widehat f_k(n)e^{2\pi int_k}e^{2\pi inu}.
\]
Let $P_{I_{k,\nu_k}}^{\mathbb T,u}$ denote the Fourier projection associated with $I_{k,\nu_k}$ acting in the $u$ variable.
By a direct calculation, we then have
\[
P_{I_{k,\nu_k}}^{\mathbb T,u}F_k(u,t_1,\ldots,t_N)
=
\bigl(P_{I_{k,\nu_k}}^{\mathbb T}f_k\bigr)(t_k+u).
\]

For every fixed \(u\), the transformation
\[
(t_1,\ldots,t_N)
\longmapsto
(t_1+u,\ldots,t_N+u)
\]
preserves Haar measure on \(\mathbb T^N\). Therefore,
\begin{align*}
&
\left\|
\sum_{k=1}^N
\varepsilon_k
P_{I_{k,\nu_k}}^{\mathbb T}f_k(t_k)
\right\|_{L^p(\Omega_\varepsilon\times\mathbb T^N;X)}
=
\left\|
\sum_{k=1}^N
\varepsilon_k
P_{I_{k,\nu_k}}^{\mathbb T,u}F_k(u,t_1,\ldots,t_N)
\right\|_{L^p(\Omega_\varepsilon\times\mathbb T^{N+1};X)}.
\end{align*}
Applying Lemma~\ref{lem:randomized-torus-intervals} in the \(u\)-variable,
with \(L^p(\mathbb T^N;X)\) as the range space, gives
\begin{align*}
&
\left\|
\sum_{k=1}^N
\varepsilon_k
P_{I_{k,\nu_k}}^{\mathbb T,u}F_k(u,t_1,\ldots,t_N)
\right\|_{L^p(\Omega_\varepsilon\times\mathbb T^{N+1};X)}
\\
&\quad\lesssim_p \beta_{p,X}^2
\left\|
\sum_{k=1}^N
\varepsilon_kF_k(u,t_1,\ldots,t_N)
\right\|_{L^p(\Omega_\varepsilon\times\mathbb T^{N+1};X)}
=\beta_{p,X}^2
\left\|
\sum_{k=1}^N
\varepsilon_kf_k(t_k)
\right\|_{L^p(\Omega_\varepsilon\times\mathbb T^N;X)}.
\end{align*}
The last equality follows from the measure-preserving transformation
\[
(u,t_1,\ldots,t_N)
\longmapsto
(u,t_1+u,\ldots,t_N+u)
\]
on \(\mathbb T^{N+1}\). This proves
\eqref{eq:fixed-random-interval-choice}.
Finally, Lemma~\ref{lem:independent-umd-randomization} gives
\[
\left\|
\sum_{k=1}^N
\varepsilon_kf_k(t_k)
\right\|_{L^p(\Omega_\varepsilon\times\mathbb T^N;X)}
\le 2
\left\|
\sum_{k=1}^N
f_k(t_k)
\right\|_{L^p(\mathbb T^N;X)}.
\]
Combining the preceding estimates proves
\eqref{eq:diagonal-bv-independent}.
\end{proof}

\subsection{Decoupling estimate for convolution operators on the cyclic groups}
\label{sec:cyclic-convolution}

We next establish the decoupling inequality for
the convolution operators on the cyclic group.
We first introduce the centered frequency representatives and the
corresponding correction multipliers.
For \(m\ge2\), let
\begin{equation*}
R_m
:=
\left\{
-\left\lfloor\frac m2\right\rfloor,
\ldots,
\left\lceil\frac m2\right\rceil-1
\right\}
\end{equation*}
be the centered complete residue system. Let
\(\rho_m:\mathbb Z_m\to R_m\) denote the centered representative map
defined by
\[
\rho_m(q)
=
\begin{cases}
q,
& 0\le q\le \left\lceil \dfrac{m}{2}\right\rceil-1,\\[6pt]
q-m,
& \left\lceil \dfrac{m}{2}\right\rceil\le q\le m-1.
\end{cases}
\]
Define
\[
\operatorname{sinc}(u)
:=
\begin{cases}
\dfrac{\sin(\pi u)}{\pi u},
&u\ne0,\\[2mm]
1,
&u=0
\end{cases}
\]
and
\[
\Phi(u)
:=
\frac1{\operatorname{sinc}(u)}
=
\frac{\pi u}{\sin(\pi u)},
\qquad |u|\le\frac12,
\]
with \(\Phi(0)=1\). We may extend \(\Phi\) periodically to \(\mathbb T\). Let
\(D_m\) be the cyclic Fourier multiplier with symbol
\[
d_m(q)
:=
\Phi\!\left(\frac{\rho_m(q)}m\right),
\qquad q\in\mathbb Z_m.
\]

The next proposition gives a uniform $\ell^1$-bound for the
convolution kernels of the multipliers $D_m$ and the resulting estimate
on independent sums.

\begin{proposition}[A decoupling estimate for cyclic convolution operators]
\label{prop:inverse-sinc-uniform-kernel}
For any $m\ge 2,$ there exists a kernel \(\kappa_m:\mathbb Z_m\to\mathbb C\) such that 
\[
D_mh(j)
=
\sum_{s\in\mathbb Z_m}\kappa_m(s)h(j-s),
\]
and
\begin{equation}\label{eq:inverse-sinc-uniform-l1}
\sup_{m\ge2}
\sum_{s\in\mathbb Z_m}|\kappa_m(s)|
<\infty.
\end{equation}
Consequently, if \(1<p<\infty\) and \(X\) is a Banach space, then,
for every \(N\ge1\), every family of integers
\((m_k)_{k=1}^N\) with \(m_k\ge2\), and every
\(h_k\in L_0^p(\mathbb Z_{m_k};X)\), one has
\begin{equation}\label{eq:decoupling-cyclic-convolution}
\left\|
\sum_{k=1}^ND_{m_k}h_k(y_k)
\right\|_{L^p\left(\prod_{k=1}^N\mathbb Z_{m_k};X\right)}
\lesssim
\left\|
\sum_{k=1}^Nh_k(y_k)
\right\|_{L^p\left(\prod_{k=1}^N\mathbb Z_{m_k};X\right)}.
\end{equation}
\end{proposition}

\begin{proof}
The periodic function \(\Phi\) is continuous and piecewise \(C^2\), with
\(\Phi'\) of bounded variation. Hence
\[
|\widehat\Phi(n)|
\lesssim
(1+|n|)^{-2},
\]
and therefore
\[
\sum_{n\in\mathbb Z}|\widehat\Phi(n)|<\infty.
\]
The periodicity of $\Phi$ and the identity
\[
\frac1m\sum_{s=0}^{m-1}e^{2\pi iqs/m}
=
\begin{cases}
1,&q=0,\\[2mm]
0,&q\ne0,
\end{cases}
\qquad q\in\mathbb Z_m,
\]
show that the inverse discrete Fourier kernel of \(d_m\) is
\[
\kappa_m(s)
=\frac1m\sum_{q=0}^{m-1}
\Phi\!\left(\frac{\rho_m(q)}m\right)e^{2\pi iqs/m}=
\frac1m\sum_{q=0}^{m-1}
\Phi\!\left(\frac qm\right)e^{2\pi iqs/m}
=
\sum_{\ell\in\mathbb Z}\widehat\Phi(\ell m-s).
\]
Consequently,
\[
\sum_{s=0}^{m-1}|\kappa_m(s)|
\le
\sum_{n\in\mathbb Z}|\widehat\Phi(n)|,
\]
which proves \eqref{eq:inverse-sinc-uniform-l1}.

Let
\(
W:=2\sum_{n\in\mathbb Z}|\widehat\Phi(n)|.
\)
As in the proof of Theorem~\ref{thm:diagonal-bv-independent}, introduce
independent random variables
\[
\mathbb S_k\in\mathbb Z_{m_k}\cup\{s_*\},
\qquad 1\le k\le N,
\]
with distribution
\[
\mathbb P(\mathbb S_k=s)
=
\frac{|\kappa_{m_k}(s)|}{W},
\qquad s\in\mathbb Z_{m_k},
\]
and
\[
\mathbb P(\mathbb S_k=s_*)
=
1-\frac1W
\sum_{s\in\mathbb Z_{m_k}}|\kappa_{m_k}(s)|.
\]
Set
\[
\theta_k(s)
:=
\begin{cases}
\dfrac{\kappa_{m_k}(s)}{|\kappa_{m_k}(s)|},
&s\in\mathbb Z_{m_k}\text{ and }\kappa_{m_k}(s)\neq0,\\[6pt]
0,
&\text{otherwise}.
\end{cases}
\]
and define the translations
\[
(\tau_sz)(j):=z(j-s),
\qquad
\tau_{s_*}:=\operatorname{Id}.
\]
We then have
\[
D_{m_k}h_k
=
W\,\mathbb E_{\mathbb S_k}
\bigl[
\theta_k(\mathbb S_k)\tau_{\mathbb S_k}h_k
\bigr],
\qquad
|\theta_k(\mathbb S_k)|\le1.
\]
Set \(\mathbb S:=(\mathbb S_1,\ldots,\mathbb S_N)\). Then Jensen's inequality and
Lemma~\ref{lem:independent-umd-randomization} give
\begin{align*}
\left\|
\sum_kD_{m_k}h_k(y_k)
\right\|_{L^p\left(\prod_{k=1}^N\mathbb Z_{m_k};X\right)}
&\le
W\,\mathbb E_{\mathbb S}
\left\|
\sum_k\theta_k(\mathbb S_k)\tau_{\mathbb S_k}h_k(y_k)
\right\|_{L^p\left(\prod_{k=1}^N\mathbb Z_{m_k};X\right)}
\\
&\lesssim W\,\mathbb E_{\mathbb S}
\left\|
\sum_k\tau_{\mathbb S_k}h_k(y_k)
\right\|_{L^p\left(\prod_{k=1}^N\mathbb Z_{m_k};X\right)}.
\end{align*}
Since independent coordinate translations preserve the joint distribution,
\[
\left\|
\sum_kD_{m_k}h_k(y_k)
\right\|_{L^p\left(\prod_{k=1}^N\mathbb Z_{m_k};X\right)}
\lesssim W
\left\|
\sum_kh_k(y_k)
\right\|_{L^p\left(\prod_{k=1}^N\mathbb Z_{m_k};X\right)}.
\]
This proves \eqref{eq:decoupling-cyclic-convolution}.
\end{proof}

\section{Sampling factorization and the CPD characterization}
\label{sec:proof-CPD}\label{s5.3}

We now use the two decoupling estimates from
Section~\ref{sec:decoupling-estimates} to prove the centered cyclic
projection inequality and thereby complete
the proof of Theorem~\ref{thm:umd-implies-CPD}. The link between these
decoupling inequalities is provided by the exact sampling factorization established below:
it expresses a centered cyclic projection in terms of a torus
multiplier, conditional expectation, and a cyclic convolution
correction. This factorization, together with
\eqref{eq:diagonal-bv-independent}, \eqref{eq:decoupling-cyclic-convolution},
and the contractivity of conditional expectation, yields the desired estimate.

We first construct the sampling factorization.
Let
\[
C_{m,j}
:=
\left[
\frac{j-\frac12}{m},
\frac{j+\frac12}{m}
\right)
\pmod 1,
\qquad j\in\mathbb Z_m,
\]
and define the centered step-function embedding
\[
(\mathcal J_mh)(t):=h(j),
\qquad t\in C_{m,j}.
\]
Then \(\mu(C_{m,j})=1/m\), and
\(\mathcal J_m:L^p(\mathbb Z_m;X)\to L^p(\mathbb T;X)\) is an
isometry.
A direct calculation gives
\begin{equation}\label{eq:step-embedding-fourier-coefficients}
\widehat{\mathcal J_mh}(n)
=
\operatorname{sinc}\!\left(\frac nm\right)
\widehat h(n\bmod m),
\qquad n\in\mathbb Z.
\end{equation}
Let \(\mathcal E_m\) denote the conditional expectation onto the
\(\sigma\)-algebra generated by the cells
\((C_{m,j})_{j=0}^{m-1}\). The range of \(\mathcal E_m\) coincides with
the range of \(\mathcal J_m\). Let \(\mathcal U_m\) denote the inverse
isometry of \(\mathcal J_m\) from the range of \(\mathcal J_m\) onto
\(L^p(\mathbb Z_m;X)\).

For an interval \(J\subset R_m\), define the 
symbol
\begin{equation*}
b_{m,J}(n)
:=
\begin{cases}
\operatorname{sinc}^{-1}(n/m),&n\in J,\\
0,&n\notin J,
\end{cases}
\qquad n\in\mathbb Z.
\end{equation*}
Indeed, if \(J=[a,b]\cap\mathbb Z\subset R_m\), then
\begin{align*}
\|b_{m,J}\|_{\operatorname{var}}
&\le
2\|\Phi\|_{L^\infty([-1/2,1/2])}
+
\sum_{n=a+1}^{b}
\left|
\Phi\left(\frac nm\right)
-
\Phi\left(\frac{n-1}{m}\right)
\right|
\\
&\le
2\|\Phi\|_{L^\infty([-1/2,1/2])}
+
\int_{-1/2}^{1/2}|\Phi'(u)|\,du.
\end{align*}
Consequently,
\begin{equation}\label{eq:uniform-bv-cyclic-symbol}
\sup_{m\ge2}
\sup_{\substack{J\subset R_m\\J\text{ an interval}}}
\|b_{m,J}\|_{\mathrm{var}}
<\infty.
\end{equation}
We also define the centered cyclic Fourier
projection by
\begin{equation}\label{def of P}
P_{m,J}h(j)
:=
\sum_{q\in J}
\widehat h(q\bmod m)e^{2\pi iqj/m}
=
\sum_{q\in J}
\widehat h(\rho_m^{-1}(q))e^{2\pi iqj/m},
\qquad j\in\mathbb Z_m.
\end{equation}
These definitions lead to an exact factorization of the cyclic
projection through a torus multiplier.

\begin{lemma}[Cyclic sampling factorization]\label{lem:exact-cyclic-sampling-factorization}
For every \(m\ge2\) and every interval \(J\subset R_m\),
\begin{equation}\label{eq:exact-cyclic-sampling-factorization}
\mathcal J_mP_{m,J}
=
\mathcal J_mD_m\mathcal U_m\mathcal E_m
T_{b_{m,J}}^{\mathbb T}\mathcal J_m.
\end{equation}
\end{lemma}

\begin{proof}
Recall that \(\mathcal E_m\) is the conditional expectation onto the
\(\sigma\)-algebra generated by the cells
\((C_{m,j})_{j=0}^{m-1}\). Thus, for an integrable function $f$,
\[
(\mathcal E_mf)(t)
=
m\int_{C_{m,j}}f(s)\,ds,
\qquad t\in C_{m,j}.
\]
Since
\[
(\mathcal J_mh)(t)=h(j),
\qquad t\in C_{m,j},
\]
if \(g\) is constant on every cell \(C_{m,j}\), then
\(\mathcal U_mg=\mathcal J_m^{-1}g\) is given by
\[
(\mathcal U_mg)(j)=g(t),
\qquad t\in C_{m,j}.
\]
Equivalently,
\[
\mathcal U_m\mathcal J_m
=
\operatorname{Id},
\qquad
\mathcal J_m\mathcal U_m\mathcal E_m
=
\mathcal E_m.
\]

For an interval \(J\subset R_m\), the torus Fourier multiplier associated
with the symbol \(b_{m,J}\) is
\[
T_{b_{m,J}}^{\mathbb T}f(t)
=
\sum_{n\in\mathbb Z}
b_{m,J}(n)\widehat f(n)e^{2\pi int}.
\]
By \eqref{eq:step-embedding-fourier-coefficients} and the definition of
\(b_{m,J}\),
\begin{align*}
T_{b_{m,J}}^{\mathbb T}\mathcal J_mh(t)
&=
\sum_{n\in J}
\frac{\widehat{\mathcal J_mh}(n)}
{\operatorname{sinc}(n/m)}e^{2\pi int}
\\
&=
\sum_{n\in J}
\widehat h(n\bmod m)e^{2\pi int}.
\end{align*}
For \(t\in C_{m,j}\), the definition of \(\mathcal E_m\) gives
\begin{align*}
(\mathcal E_mT_{b_{m,J}}^{\mathbb T}\mathcal J_mh)(t)
&=
m\int_{C_{m,j}}T_{b_{m,J}}^{\mathbb T}\mathcal J_mh(s)\,ds
\\
&=
\sum_{n\in J}
\widehat h(n\bmod m)
m\int_{C_{m,j}}e^{2\pi ins}\,ds
\\
&=
\sum_{n\in J}
\operatorname{sinc}\!\left(\frac{n}{m}\right)
\widehat h(n\bmod m)e^{2\pi inj/m}.
\end{align*}
It follows from the definition of \(\mathcal U_m\) that
\[
\bigl(\mathcal U_m\mathcal E_mT_{b_{m,J}}^{\mathbb T}\mathcal J_mh\bigr)(j)
=
\sum_{n\in J}
\operatorname{sinc}\!\left(\frac{n}{m}\right)
\widehat h(n\bmod m)e^{2\pi inj/m}.
\]
Since \(J\subset R_m\), for every \(n\in J\),
\[
\rho_m(n\bmod m)=n.
\]
Therefore,
\[
d_m(n\bmod m)
=
\Phi\!\left(\frac{n}{m}\right)
=
\frac1{\operatorname{sinc}(n/m)}.
\]
Applying \(D_m\) removes the factor
\(\operatorname{sinc}(n/m)\), and hence
\begin{align*}
\bigl(D_m\mathcal U_m\mathcal E_mT_{b_{m,J}}^{\mathbb T}\mathcal J_mh\bigr)(j)
&=
\sum_{n\in J}
\widehat h(n\bmod m)e^{2\pi inj/m}
\\
&=
P_{m,J}h(j).
\end{align*}
Finally, applying \(\mathcal J_m\) to both sides yields
\[
\mathcal J_mD_m\mathcal U_m\mathcal E_m
T_{b_{m,J}}^{\mathbb T}\mathcal J_mh
=
\mathcal J_mP_{m,J}h,
\]
which proves \eqref{eq:exact-cyclic-sampling-factorization}.
\end{proof}

The factorization in
Lemma~\ref{lem:exact-cyclic-sampling-factorization} now links the two
decoupling inequalities proved in
Section~\ref{sec:decoupling-estimates} and yields the following decoupling inequality for centered
cyclic Fourier projections.
\begin{proposition}[A decoupling estimate for centered cyclic projections]
\label{prop:independent-centered-cyclic-intervals}
Let \(1<p<\infty\), and let \(X\) be a \(\operatorname{UMD}\) Banach space. For every
\(N\ge1\), every family of integers \((m_k)_{k=1}^N\) with
\(m_k\ge2\), every family of intervals \(J_k\subset R_{m_k}\), every
\(h_k\in L^p_0(\mathbb Z_{m_k};X)\), and every scalar family
\((\theta_k)_{k=1}^N\) satisfying \(|\theta_k|\le1\),
\begin{equation}\label{eq:independent-centered-cyclic-intervals}
\left\|
\sum_{k=1}^N
\theta_kP_{m_k,J_k}h_k(y_k)
\right\|_{L^p\left(\prod_{k=1}^N\mathbb Z_{m_k};X\right)}
\lesssim_p \beta_{p,X}^2
\left\|
\sum_{k=1}^Nh_k(y_k)
\right\|_{L^p\left(\prod_{k=1}^N\mathbb Z_{m_k};X\right)}.
\end{equation}
\end{proposition}

\begin{proof}
Since \(h_k\) has mean zero and \(\mu(C_{m_k,j})=1/m_k\),
\[
\int_{\mathbb T}\mathcal J_{m_k}h_k\,d\mu
=
\frac1{m_k}\sum_{j\in\mathbb Z_{m_k}}h_k(j)
=
0.
\]
Set
\[
g_k
:=
\theta_kT_{b_{m_k,J_k}}^{\mathbb T}\mathcal J_{m_k}h_k
=
T_{\theta_kb_{m_k,J_k}}^{\mathbb T}\mathcal J_{m_k}h_k,
\qquad
f_k
:=
\mathcal U_{m_k}\mathcal E_{m_k}g_k.
\]
The multiplier \(T_{b_{m_k,J_k}}^{\mathbb T}\), the conditional
expectation \(\mathcal E_{m_k}\), and the identification
\(\mathcal U_{m_k}\) all map mean-zero functions to mean-zero functions.
Consequently, each \(f_k\) has mean zero. The factorization
\eqref{eq:exact-cyclic-sampling-factorization}, followed by
Proposition~\ref{prop:inverse-sinc-uniform-kernel}, gives
\begin{align*}
\left\|
\sum_k\theta_kP_{m_k,J_k}h_k(y_k)
\right\|_{L^p\left(\prod_{k=1}^N\mathbb Z_{m_k};X\right)}
&=
\left\|
\sum_k\theta_k\mathcal J_{m_k}P_{m_k,J_k}h_k(t_k)
\right\|_{L^p(\mathbb T^N;X)}
\\
&=
\left\|
\sum_k\mathcal J_{m_k}D_{m_k}f_k(t_k)
\right\|_{L^p(\mathbb T^N;X)}\\
&\lesssim
\left\|
\sum_k\mathcal J_{m_k}f_k(t_k)
\right\|_{L^p(\mathbb T^N;X)}
=
\left\|
\sum_k\mathcal E_{m_k}g_k(t_k)
\right\|_{L^p(\mathbb T^N;X)}.
\end{align*}
Here we use the fact that each
\(\mathcal J_{m_k}:L^p(\mathbb Z_{m_k};X)\to L^p(\mathbb T;X)\) is an
isometry acting in an independent coordinate \(t_k\), and $\mathcal J_{m_k}\mathcal U_{m_k}\mathcal E_{m_k}=\mathcal E_{m_k}$.

Let \(\mathcal A_{\mathrm{cell}}\) denote the product of the cell
\(\sigma\)-algebras:
\[
\mathcal A_{\mathrm{cell}}
:=
\bigotimes_{k=1}^N
\sigma\left(\{C_{m_k,j}\}_{j=0}^{m_k-1}\right).
\]
Since \(g_k\) depends only on \(t_k\),
\[
\sum_k\mathcal E_{m_k}g_k(t_k)
=
\mathbb E\left(
\sum_kg_k(t_k)\,\middle|\,\mathcal A_{\mathrm{cell}}
\right).
\]
Hence, by contractivity of conditional expectation,
\[
\left\|
\sum_k\mathcal E_{m_k}g_k(t_k)
\right\|_{L^p(\mathbb T^N;X)}
\le
\left\|
\sum_kg_k(t_k)
\right\|_{L^p(\mathbb T^N;X)}.
\]
By \eqref{eq:uniform-bv-cyclic-symbol}, the symbols
\(\theta_kb_{m_k,J_k}\) satisfy the hypotheses of
Theorem~\ref{thm:diagonal-bv-independent}. Therefore,
\[
\left\|
\sum_kg_k(t_k)
\right\|_{L^p(\mathbb T^N;X)}
\lesssim_p \beta_{p,X}^2
\left\|
\sum_k\mathcal J_{m_k}h_k(t_k)
\right\|_{L^p(\mathbb T^N;X)}.
\]
The last norm equals the norm on the right-hand side of
\eqref{eq:independent-centered-cyclic-intervals}, since the embeddings
\(\mathcal J_{m_k}\) act isometrically in independent coordinates.
\end{proof}

With the above decoupling inequality, we can now prove Theorem \ref{thm:umd-implies-CPD}.

\begin{proof}[Proof of Theorem~\ref{thm:umd-implies-CPD}]
We first show that the \(\operatorname{UMD}\) property implies the \(\operatorname{CPD}_p\) property.
Assume that \(X\) is \(\operatorname{UMD}\). Let
\[
\mathcal K_0
:=
\left\{
k:
0\le a_k\le\left\lfloor\frac{m_k}{2}\right\rfloor
\right\},
\qquad
\mathcal K_1
:=
\{1,\ldots,N\}\setminus\mathcal K_0.
\]
If \(k\in\mathcal K_0\), then from the definitions \eqref{def of P} and \eqref{def of Q},
\[
Q_{m_k,a_k}
=
P_{m_k,\{-a_k,\ldots,-1\}}.
\]
Here and below, the interval \(\{-a_k,\ldots,-1\}\) is understood to be
empty when \(a_k=0\).
If \(k\in\mathcal K_1\), set
\(
b_k:=m_k-a_k-1.
\)
Then \(0\le b_k<m_k/2\), and, on mean-zero functions,
\(P_{m_k,\{0\}}=0\). Consequently,
\begin{equation*}
Q_{m_k,a_k}
=
\operatorname{Id}-P_{m_k,\{1,\ldots,b_k\}}.
\end{equation*}
The interval \(\{1,\ldots,b_k\}\) is understood to be empty when
\(b_k=0\). Therefore,
\begin{align*}
\sum_{k=1}^NQ_{m_k,a_k}h_k
&=
\sum_{k\in\mathcal K_1}h_k
+
\sum_{k\in\mathcal K_0}
P_{m_k,\{-a_k,\ldots,-1\}}h_k
\notag\\
&\qquad-
\sum_{k\in\mathcal K_1}
P_{m_k,\{1,\ldots,b_k\}}h_k.
\end{align*}
The first sum on the right-hand side is the conditional expectation of
\(\sum_{k=1}^Nh_k(y_k)\) with respect to
\[
\mathcal A_{\mathcal K_1}
:=
\sigma(y_k:k\in\mathcal K_1).
\]
Its norm is therefore bounded by the norm on the right-hand side of
\eqref{eq:CPD-property}. The remaining two sums can be combined into a
single family of centered cyclic interval projections with coefficients
\(\theta_k\in\{-1,1\}\). Proposition~\ref{prop:independent-centered-cyclic-intervals}
therefore proves \eqref{eq:CPD-property} with $\gamma_{p,X}\lesssim_p \beta_{p,X}^2$.

We then show that the \(\operatorname{CPD}_p\) property implies the \(\operatorname{UMD}\) property.
Assume that \(X\) has the \(\operatorname{CPD}_p\) property. Let
\[
h(t)=\sum_{|n|\le L}x_ne^{2\pi int}
\]
be an \(X\)-valued trigonometric polynomial, where \(L\ge1\). For an
integer \(m>2L\), define
\[
g_m(j):=h(j/m)-x_0,\qquad j\in\mathbb Z_m.
\]
It is easy to verify that
\(g_m\in L_0^p(\mathbb Z_m;X)\) and
\[
Q_{m,L}g_m(j)
=\sum_{n=-L}^{-1}x_ne^{2\pi inj/m}
=(P_-^{\mathbb T}h)(j/m)
\]
by \eqref{def of Q},
where \(P_-^{\mathbb T}:=\operatorname{Id}-P_+^{\mathbb T}\) and $P_+^{\mathbb T}$ denotes the Riesz projection. Applying \eqref{eq:CPD-property} with \(N=1\) gives
\[
\left(\frac1m\sum_{j=0}^{m-1}
\|(P_-^{\mathbb T}h)(j/m)\|_X^p\right)^{1/p}
\le \gamma_{p,X}
\left(\frac1m\sum_{j=0}^{m-1}
\|h(j/m)-x_0\|_X^p\right)^{1/p}.
\]
Letting \(m\to\infty\) and using convergence of the Riemann sums, we obtain
\[
\|P_-^{\mathbb T}h\|_{L^p(\mathbb T;X)}
\le \gamma_{p,X}\|h-x_0\|_{L^p(\mathbb T;X)}
\le 2\gamma_{p,X}\|h\|_{L^p(\mathbb T;X)}.
\]
Since \(P_+^{\mathbb T}h=h-P_-^{\mathbb T}h\), it follows that
\[
\|P_+^{\mathbb T}h\|_{L^p(\mathbb T;X)}
\le (1+2\gamma_{p,X})\|h\|_{L^p(\mathbb T;X)}.
\]
By density, \(P_+^{\mathbb T}\) extends to a bounded operator on
\(L^p(\mathbb T;X)\). The Riesz-projection characterization in
Remark~\ref{Riesz} implies that \(X\) is UMD with $\beta_{p,X}\lesssim_p \gamma_{p,X}^2$.
\end{proof}

\section{Proof of Corollary~\ref{cor:umd-partial-sums-schauder}}
\label{sec:schauder}

With Theorem~\ref{thm:uniform-umd-vilenkin-partial-sums},
we now turn to the equivalence between the UMD property,
the uniform boundedness of the partial-sum operators, and
the Schauder-decomposition property for a fixed Vilenkin system.
The proof follows essentially the same lines as in the bounded case.
We first recall the terminology needed for the statement of the corollary and for the discussion of block decompositions in the introduction.

\begin{definition}[Schauder basis and Schauder decomposition]
Let \(X\) be a Banach space. A sequence \((e_j)_{j\ge0}\) in \(X\) is
a \emph{Schauder basis} if, for every \(x\in X\), there exists a
unique sequence of scalars \((a_j)_{j\ge0}\) such that
\(
x=\sum_{j=0}^\infty a_je_j,
\)
where the series converges in the norm of \(X\).
More generally, a sequence \((X_j)_{j\ge0}\) of closed subspaces of
\(X\) is a \emph{Schauder decomposition} if every \(x\in X\) admits a
unique representation
\[
x=\sum_{j=0}^\infty x_j,
\qquad x_j\in X_j,
\]
with convergence in the norm of \(X\). A Schauder basis
\((e_j)_{j\ge0}\) induces the one-dimensional Schauder decomposition
\((\mathbb Ce_j)_{j\ge0}\).
\end{definition}

To describe convergence in terms of operators, suppose that \(X\)
admits a Schauder decomposition \((X_n)_{n\ge0}\).
The associated coordinate projections are denoted
by \(\mathsf D_n\), \(n\ge0\), and are given by
\[
\mathsf D_n:X\to X_n,
\qquad
\mathsf D_nx:=x_n
\quad\text{whenever}\quad
x=\sum_{j=0}^\infty x_j.
\]

The discussion of coarse and fine blockings in the introduction also
uses the stronger notion of unconditionality, which we recall for
completeness.

\begin{definition}[Unconditionality]
A Schauder decomposition \((X_n)_{n\ge0}\), with coordinate
projections \((\mathsf D_n)_{n\ge0}\), is \emph{unconditional} if,
for every \(x\in X\), the series
\(
\sum_{n\ge0}\mathsf D_nx
\)
converges unconditionally to \(x\). Equivalently, there exists a
constant \(K<\infty\) such that
\begin{equation}\label{eq:unconditional-decomposition-definition}
\left\|
\sum_{n\in F}\theta_n\mathsf D_nx
\right\|_X
\le K\|x\|_X
\end{equation}
for every finite set \(F\subset\mathbb N\), every choice of signs
\(\theta_n\in\{-1,1\}\), and every \(x\in X\). The least admissible
constant \(K\) is called the \emph{unconditional constant} of the
decomposition. A Schauder basis is called unconditional if its
associated one-dimensional Schauder decomposition \((\mathbb Ce_j)_{j\ge0}\) is unconditional.
\end{definition}

We now prove Corollary~\ref{cor:umd-partial-sums-schauder}.
The implication from UMD to uniform boundedness is furnished by
Theorem~\ref{thm:uniform-umd-vilenkin-partial-sums}.
The equivalence with the Schauder-decomposition property follows from
density of the Vilenkin polynomials and the uniform boundedness
principle. To recover the UMD property, we use the Paley conjugation
identity to obtain bounds for martingale transforms.
\begin{proof}[Proof of
Corollary~\ref{cor:umd-partial-sums-schauder}]
The implication
\(\mathrm{(i)}\Rightarrow\mathrm{(ii)}\), including uniformity in
\(\mathbf m\), follows directly from
Theorem~\ref{thm:uniform-umd-vilenkin-partial-sums}.
To prove
\(\mathrm{(ii)}\Rightarrow\mathrm{(iii)}\), set
\[
C:=
\sup_{n\ge1}
\|S_n\|_{L^p(\Gm;X)\to L^p(\Gm;X)}
<\infty.
\]
\(X\)-valued Vilenkin polynomials are dense in
\(L^p(\Gm;X)\). Given \(f\in L^p(\Gm;X)\) and \(\varepsilon>0\),
choose an \(X\)-valued Vilenkin polynomial \(g\) such that
\[
\|f-g\|_{L^p(\Gm;X)}<\varepsilon.
\]
For all sufficiently large \(n\), we have \(S_ng=g\), and hence
\[
\begin{aligned}
\|S_nf-f\|_{L^p(\Gm;X)}
&\le
\|S_n(f-g)\|_{L^p(\Gm;X)}
+\|g-f\|_{L^p(\Gm;X)}
\\
&\le
C\|f-g\|_{L^p(\Gm;X)}
+\|g-f\|_{L^p(\Gm;X)}
\\
&<
(C+1)\varepsilon.
\end{aligned}
\]
Thus \(S_nf\to f\) in \(L^p(\Gm;X)\). Equivalently,
\(
f=\sum_{n=0}^\infty\widehat f(n)\psi_n
\)
with convergence in \(L^p(\Gm;X)\).
Uniqueness of this representation follows from the uniqueness
of the Fourier--Vilenkin coefficients.
Therefore, \((\psi_nX)_{n\ge0}\) is a Schauder decomposition
of \(L^p(\Gm;X)\).

Conversely, if \(\mathrm{(iii)}\) holds, then the partial-sum
projections of this Schauder decomposition are precisely the
operators \(S_n\). In particular, \(S_nf\to f\) for every
\(f\in L^p(\Gm;X)\), and the Banach--Steinhaus theorem yields
\[
\sup_{n\ge1}
\|S_n\|_{L^p(\Gm;X)\to L^p(\Gm;X)}
<\infty.
\]
This proves
\(\mathrm{(iii)}\Rightarrow\mathrm{(ii)}\).

It remains to prove
\(\mathrm{(ii)}\Rightarrow\mathrm{(i)}\).
Let \(J\ge0\), choose
\(\eta_0,\ldots,\eta_J\in\{0,1\}\), and set
\[
n_\eta
:=
\sum_{j=0}^J\eta_j(m_j-1)M_j.
\]
We adopt the convention \(S_0:=0\).
With $B_{n_\eta}$ and $\pi_{j,a}$ defined as in
Section~\ref{s2.2}, we have
\[
P_{B_{n_\eta}}
=
\sum_{j=0}^J
\pi_{j,\eta_j(m_j-1)}
=
\sum_{j=0}^J\eta_jd_j.
\]
The Paley conjugation identity
\eqref{eq:paley-conjugation-operator} therefore gives
\begin{equation}\label{eq:weisz-partial-sum-identity}
\overline{\psi_{n_\eta}}\,
S_{n_\eta}\bigl(\psi_{n_\eta}f\bigr)
=
\sum_{j=0}^J\eta_jd_jf.
\end{equation}
Since multiplication by a Vilenkin character is an isometry on
\(L^p(\Gm;X)\), it follows from
\eqref{eq:partial-sum-uniform-corrected} that
\begin{equation}\label{eq:zero-one-martingale-transform}
\left\|
\sum_{j=0}^J\eta_jd_jf
\right\|_{L^p(\Gm;X)}
\le
C\|f\|_{L^p(\Gm;X)}.
\end{equation}

Now let \(\theta_j\in\{-1,1\}\) and set
\(
\eta_j:=\frac{1+\theta_j}{2}.
\)
Since
\[
\sum_{j=0}^J\theta_jd_jf
=
2\sum_{j=0}^J\eta_jd_jf
-
\sum_{j=0}^Jd_jf
=
2\sum_{j=0}^J\eta_jd_jf
-
(E_{J+1}-E_0)f,
\]
we obtain
\[
\begin{aligned}
\left\|
\sum_{j=0}^J\theta_jd_jf
\right\|_{L^p(\Gm;X)}
&\le
2\left\|
\sum_{j=0}^J\eta_jd_jf
\right\|_{L^p(\Gm;X)}
+
\|(E_{J+1}-E_0)f\|_{L^p(\Gm;X)}
\\
&\le
2C\|f\|_{L^p(\Gm;X)}
+
2\|f\|_{L^p(\Gm;X)}
\\
&=
2(C+1)\|f\|_{L^p(\Gm;X)}.
\end{aligned}
\]
By the characterization of UMD spaces in terms of the atomic
Vilenkin filtration recalled in \cite[p.~416]{Weisz2007},
this establishes \(\mathrm{(ii)}\Rightarrow\mathrm{(i)}\)
and completes the proof.
\end{proof}

\section{\texorpdfstring{$\mathcal R$}{R}-boundedness of the fine-block partial sums}
\label{sec:R-bound}

We now prove Theorem~\ref{thm:R-bound} by a Rademacher-valued
analogue of the reduction developed in
Sections~\ref{sec:CPD-to-partial-sums},
\ref{sec:decoupling-estimates}, and~\ref{sec:proof-CPD}. The argument differs from the earlier proofs in two respects:
\begin{enumerate}
\item[\rm{(i)}] The estimate defining $\mathcal R$-boundedness already
involves Rademacher randomization, so the additional randomization
steps used in Sections~\ref{sec:decoupling-estimates}
and~\ref{s5.3} are unnecessary. Consequently, the decoupling
estimate on the torus may be combined directly with the sampling
factorization to obtain the desired estimate for centered cyclic
projections, without invoking the cyclic decoupling convolution estimate
of Proposition~\ref{prop:inverse-sinc-uniform-kernel}.

\item[\rm{(ii)}]  The coordinate projections arising from the fine
blocking can be expressed as sums of two centered cyclic
projections; see \eqref{coincide1}. Therefore, no separate decoupling estimate for terminal
cyclic projections of the type appearing in
\eqref{eq:CPD-property} is needed, and the proof can be completed
directly from the centered cyclic projection estimate.
\end{enumerate}

We begin with a randomized analogue of the centered cyclic
projection estimate in
Proposition~\ref{prop:independent-centered-cyclic-intervals}.
We retain the notation \(R_m\) for the centered complete residue
system introduced in Section~\ref{sec:cyclic-convolution}
and \(P_{m,J}\) for the centered cyclic Fourier projection
defined in Section~\ref{s5.3}.

\begin{lemma}\label{lem:R-independent-cyclic}
Let \(X\) be a \(\operatorname{UMD}\) Banach space and \(1<p<\infty\).
For \(a=1,\ldots,N\), let
\(k_a\in\{0,\ldots,K\}\), let
\(I_a\subset R_{m_{k_a}}\) be an interval, and let
\(
h_a\in L^p(\mathbb Z_{m_{k_a}};X).
\)
Then
\begin{equation}\label{eq:R-independent-cyclic}
\begin{aligned}
\left\|
 \sum_{a=1}^N
 \varepsilon_a
 P_{m_{k_a},I_a}h_a(y_{k_a})
\right\|_{
 L^p\left(
 \Omega_\varepsilon
 \times\prod_{k=0}^K\mathbb Z_{m_k};X
 \right)}
\lesssim_p\beta_{p,X}^2
\left\|
 \sum_{a=1}^N
 \varepsilon_a h_a(y_{k_a})
\right\|_{
 L^p\left(
 \Omega_\varepsilon
 \times\prod_{k=0}^K\mathbb Z_{m_k};X
 \right)} .
\end{aligned}
\end{equation}
Here the indices $k_a$ need not be distinct and the functions $h_a$ need not have mean zero.
\end{lemma}

\begin{proof}

We first derive a randomized version of
\eqref{eq:diagonal-bv-independent} from
Lemma~\ref{lem:randomized-torus-intervals}.
Let \(Y\) be a UMD space, let
\(F_a\in L^p(\mathbb T;Y)\), and let
\(\mathfrak b_a:\mathbb Z\to\mathbb C\) be finitely supported
symbols satisfying
\[
\sup_a\|\mathfrak b_a\|_{\operatorname{var}}\le V.
\]
Choose integers \(u_a\le v_a\) such that
\(\operatorname{supp}\mathfrak b_a\subseteq [u_a,v_a]\).
As in the proof of Theorem~\ref{thm:diagonal-bv-independent},
telescoping gives the interval decomposition
\[
\mathfrak b_a(n)
=
\sum_{r=u_a}^{v_a}
\bigl(\mathfrak b_a(r)-\mathfrak b_a(r-1)\bigr)
\mathbf 1_{[r,v_a]}(n).
\]
The sum of the absolute values of the coefficients in this
decomposition is bounded by the variation of the symbol.
Thus, Lemma~\ref{lem:randomized-torus-intervals}, together with
Proposition~\ref{lem:R}{\rm (ii)}, yields
\begin{equation}\label{eq:diagonal-bv-independent-general}
\left\|
\sum_a\varepsilon_aT_{\mathfrak b_a}^{\mathbb T}F_a
\right\|_{L^p(\Omega_\varepsilon\times\mathbb T;Y)}
\lesssim_p\beta_{p,Y}^2
V
\left\|
\sum_a\varepsilon_aF_a
\right\|_{L^p(\Omega_\varepsilon\times\mathbb T;Y)}.
\end{equation}

We next combine the argument based on a common translation
from the proof of Theorem~\ref{thm:diagonal-bv-independent}
with the sampling factorization of
Lemma~\ref{lem:exact-cyclic-sampling-factorization}.
Let $C_{m,j}$ and $\mathcal J_m$ be as in Section~\ref{s5.3}.
For \(a=1,\ldots,N\), we set
\[
\widetilde H_a(u,t)
:=
(\mathcal J_{m_{k_a}}h_a)(t_{k_a}+u),
\quad
u\in\mathbb T,\quad t=(t_0,\ldots,t_K)\in\mathbb T^{K+1}.
\]
For fixed \(u\in\mathbb T\) and \(k=0,\dots,K\),
let \(\mathcal E_k^{(u)}\) be the conditional expectation onto
the $\sigma$-algebra generated by the shifted partition
\[
\{t_k\in\mathbb T:\quad t_k+u\in C_{m_k,j}\},\quad j\in\mathbb Z_{m_k}
\]
of $\mathbb T,$ and set
$$\mathcal E^{(u)}:=\bigotimes_{k=0}^K \mathcal E^{(u)}_k:\quad L^p(\mathbb T^{K+1};X)\to L^p(\mathbb T^{K+1};X).$$
Put
\(
w(s):=\operatorname{sinc}^{-2}(s), |s|\le \frac12,
\)
and, for an interval \(I\subset R_m\), define the finitely
supported symbol
\[
\widetilde b_{m,I}(n):=
\begin{cases}
\mathbf 1_I(n)\,w(n/m),&n\in R_m,\\
0,&n\notin R_m.
\end{cases}
\]
Since \(w\in C^1([-1/2,1/2])\), there is a constant \(V'<\infty\),
independent of \(m\) and \(I\), such that
\begin{equation}\label{var}
\sup_{m\ge2}
\sup_{\substack{I\subset R_m\\ I\ \mathrm{interval}}}
\|\widetilde b_{m,I}\|_{\operatorname{var}}
\le V'.
\end{equation}

Let \(T_{\widetilde b_{m,I}}^{\mathbb T,u}\) denote the Fourier
multiplier acting only in the \(u\)-variable, and set
$\mathfrak b_{a}=\widetilde b_{m_{k_a},I_a}$.
We now adapt the calculation of Fourier coefficients and cell
averages from the proof of
Lemma~\ref{lem:exact-cyclic-sampling-factorization}
to the shifted cells. When \(t_{k_a}+u\in C_{m_{k_a},j}\),
the relevant cell average is given by
\[
m_{k_a}
\int_{C_{m_{k_a},j}-u}
e^{2\pi in(s+u)}\,ds
=
\operatorname{sinc}\left(\frac{n}{m_{k_a}}\right)
e^{2\pi inj/m_{k_a}}.
\]
Combining this identity with
\eqref{eq:step-embedding-fourier-coefficients} and the definition
of the multiplier symbol gives
\begin{equation}\label{cal}
    \begin{aligned}
\mathcal E^{(u)}
T_{\mathfrak b_{a}}^{\mathbb T,u}\widetilde H_a(u,t)
&=
\sum_{n\in I_a}
\mathfrak b_{a}(n)
\operatorname{sinc}^{2}\!\left(\frac{n}{m_{k_a}}\right)
\widehat h_a(n\bmod m_{k_a})
e^{2\pi inj/m_{k_a}}
\\
&=
\sum_{n\in I_a}
\widehat h_a(n\bmod m_{k_a})
e^{2\pi inj/m_{k_a}}
=
\bigl(P_{m_{k_a},I_a}h_a\bigr)(j).
\end{aligned}
\end{equation}
Compared with Lemma~\ref{lem:exact-cyclic-sampling-factorization}, the inverse-sinc correction underlying Proposition~\ref{prop:inverse-sinc-uniform-kernel} is incorporated here into the torus multiplier. 

Set
\(
Y:=L^p(\mathbb T^{K+1};X).
\)
Using $\mu(C_{m_{k_a},j})=1/m_{k_a},$ \eqref{cal}, and
the contractivity of $\mathcal E^{(u)},$ and then applying
\eqref{eq:diagonal-bv-independent-general} with the uniform
bound \eqref{var} to the functions
\(u\mapsto\widetilde H_a(u,\cdot)\), we obtain
\[
\begin{aligned}
\left\|
 \sum_{a=1}^N\varepsilon_a
 \bigl(P_{m_{k_a},I_a}h_a\bigr)(y_{k_a})
\right\|_{
 L^p\left(
 \Omega_\varepsilon\times\prod_{k=0}^K\mathbb Z_{m_k};X
 \right)}
&=
\left\|
 \sum_{a=1}^N\varepsilon_a
 \mathcal E^{(u)}
 T_{\mathfrak b_a}^{\mathbb T,u}\widetilde H_a(u,t)
\right\|_{
 L^p\left(
 \Omega_\varepsilon\times\mathbb T^{K+1}\times\mathbb T;X
 \right)}
\\
&\le
\left\|
 \sum_{a=1}^N\varepsilon_a
 T_{\mathfrak b_a}^{\mathbb T,u}\widetilde H_a(u,t)
\right\|_{
 L^p\left(
 \Omega_\varepsilon\times\mathbb T^{K+1}\times\mathbb T;X
 \right)}
\\
&\lesssim_p\beta_{p,X}^2
\left\|
 \sum_{a=1}^N\varepsilon_a \widetilde H_a(u,t)
\right\|_{
 L^p\left(
 \Omega_\varepsilon\times\mathbb T^{K+1}\times\mathbb T;X
 \right)}\\
 &=\beta_{p,X}^2
\left\|
 \sum_{a=1}^N\varepsilon_a \mathcal{J}_{m_{k_a}}h_a(t_{k_a})
\right\|_{
 L^p\left(
 \Omega_\varepsilon\times\mathbb T^{K+1};X
 \right)}.
\end{aligned}
\]
The last equality follows from the same argument using
translation invariance as in the proof of
Theorem~\ref{thm:diagonal-bv-independent}: for each fixed \(u\),
the transformation
\[
(t_0,\ldots,t_K)
\longmapsto
(t_0+u,\ldots,t_K+u)
\]
preserves Haar measure. Finally, the isometric property of
\(\mathcal J_{m_k}\) gives
\[
\left\|
 \sum_{a=1}^N\varepsilon_a \mathcal{J}_{m_{k_a}}h_a(t_{k_a})
\right\|_{
 L^p\left(
 \Omega_\varepsilon\times\mathbb T^{K+1};X
 \right)}
=
\left\|
 \sum_{a=1}^N\varepsilon_a h_a(y_{k_a})
\right\|_{
 L^p\left(
 \Omega_\varepsilon\times\prod_{k=0}^K\mathbb Z_{m_k};X
 \right)}.
\]
This proves \eqref{eq:R-independent-cyclic}.

\end{proof}

We next pass from independent cyclic coordinates to the adapted
fine-block projections. Applying the decoupling argument of
Proposition~\ref{prop:adapted-terminal-core} to the randomized estimate
in Lemma~\ref{lem:R-independent-cyclic} gives the following bound.
\begin{lemma}\label{lem:R-adapted-terminal}
Let \(X\) be a \(\operatorname{UMD}\) Banach space and let \(1<p<\infty\).
For \(a=1,\ldots,N\), let
\(k_a\in\{0,\ldots,K\}\), \(1\le j_a<m_{k_a}\), and
\(h_a\in L^p(\Gm;X)\). Then
\begin{equation}\label{eq:R-adapted-terminal}
\left\|
 \sum_{a=1}^N\varepsilon_a
 \sum_{i=1}^{j_a-1}\Delta_{k_a,i}h_a
\right\|_{L^p(\Omega_\varepsilon\times\Gm;X)}
\lesssim_{p}\beta_{p,X}^4
\left\|
 \sum_{a=1}^N\varepsilon_a d_{k_a}h_a
\right\|_{L^p(\Omega_\varepsilon\times\Gm;X)}.
\end{equation}
\end{lemma}

\begin{proof}
We follow the proof of
Proposition~\ref{prop:adapted-terminal-core} in
Section~\ref{sec:CPD-to-partial-sums}, adapting its construction
to the present randomized setting.
Define the two (possibly empty) intervals
\[
I_a^+:=R_{m_{k_a}}\cap[1,j_a),
\qquad
I_a^-:=R_{m_{k_a}}\cap[1-m_{k_a},j_a-m_{k_a}),
\]
and put
\[
A_a:=P_{m_{k_a},I_a^+}+P_{m_{k_a},I_a^-}.
\]
The same coordinatewise Fourier calculation as in the proof of
\eqref{coincide} yields
\begin{equation}\label{coincide1}
\sum_{i=1}^{j_a-1}\Delta_{k_a,i}h_a
=A_a^{(x_{k_a})}d_{k_a}h_a.
\end{equation}

Write
\[
g_a(x_{<k_a},x_{k_a}):=(d_{k_a}h_a)(x).
\]
By the same separability reduction as in the proof of
Proposition~\ref{prop:adapted-terminal-core}, we may work in a
separable closed subspace of \(X\).
On the enlarged probability space
\[
\widehat\Omega:=\Gm\times\prod_{k=0}^K\mathbb Z_{m_k},
\]
consider the filtration with
\(\widehat{\mathcal F}_0=\{\varnothing,\widehat\Omega\}\) and
\(\widehat{\mathcal F}_{k+1}=\sigma(x_0,\ldots,x_k,y_0,\ldots,y_k)\)
for \(0\le k\le K\), and set \(\mathcal G=\sigma(x_0,x_1,\ldots)\).
For each fixed choice of the Rademacher signs, group the summands
according to their common coordinate index and define
\[
\begin{aligned}
D_k^\varepsilon(x)
&:=\sum_{\{a:k_a=k\}}\varepsilon_a g_a(x_{<k},x_k),\\
\widetilde D_k^\varepsilon(x)
&:=\sum_{\{a:k_a=k\}}\varepsilon_a
  \bigl(A_a g_a(x_{<k},\cdot)\bigr)(x_k).
\end{aligned}
\]
Replacing the current coordinate by its independent copy,
we define the corresponding sequences
\[
\begin{aligned}
D_k^{\varepsilon\prime}(x,y)
&:=\sum_{\{a:k_a=k\}}\varepsilon_a g_a(x_{<k},y_k),\\
\widetilde D_k^{\varepsilon\prime}(x,y)
&:=\sum_{\{a:k_a=k\}}\varepsilon_a
  \bigl(A_a g_a(x_{<k},\cdot)\bigr)(y_k).
\end{aligned}
\]
For each fixed realization of the outer Rademacher sequence
\((\varepsilon_a)_a\), the sequences
\(\{D_k^\varepsilon\}_k\) and
\(\{\widetilde D_k^\varepsilon\}_k\) play the roles of
\(\{D_k\}_k\) and \(\{\widetilde D_k\}_k\), respectively, in the
earlier argument. Similarly,
\(\{D_k^{\varepsilon\prime}\}_k\) and
\(\{\widetilde D_k^{\varepsilon\prime}\}_k\) are the counterparts of
\(\{D_k'\}_k\) and \(\{\widetilde D_k'\}_k\) from the proof of
Proposition~\ref{prop:adapted-terminal-core}.
Applying the calculations in Steps~2 and~3 of that proof, we similarly show that
$\{D_k^\varepsilon\}_k,\{\widetilde D_k^\varepsilon\}_k$
are martingale difference sequences, with
$\{D_k^{\varepsilon\prime}\}_k,\{\widetilde D_k^{\varepsilon\prime}\}_k$
as their respective decoupled tangent sequences.

Now we use the abbreviated notation
\[
\|\cdot\|_{L^{p}_{\varepsilon,x}}:=
\|\cdot\|_{L^p(\Omega_\varepsilon\times\Gm;X)},
\qquad
\|\cdot\|_{L^{p}_{\varepsilon,x,y}}:=
\|\cdot\|_{L^p(\Omega_\varepsilon\times\widehat\Omega;X)},
\]
and proceed as in Step~4 of the proof of
Proposition~\ref{prop:adapted-terminal-core}.
Applying \eqref{eq:umd-tangent-decoupling} for each fixed choice
of the signs and 
Lemma~\ref{lem:R-independent-cyclic} separately to \(I_a^+\) and
\(I_a^-\), we deduce that
\begin{eqnarray*}
\left\|\sum_a\varepsilon_a\sum_{i=1}^{j_a-1}
\Delta_{k_a,i}h_a\right\|_{L^{p}_{\varepsilon,x}}
&\overset{\eqref{coincide1}}{=}&
\left\|\sum_k\widetilde D_k^\varepsilon\right\|_{L^{p}_{\varepsilon,x}}
\overset{\eqref{eq:umd-tangent-decoupling}}{\le}\beta_{p,X}
\left\|\sum_k\widetilde D_k^{\varepsilon\prime}
\right\|_{L^{p}_{\varepsilon,x,y}}\\
&\overset{\eqref{eq:R-independent-cyclic}}{\lesssim_{p}}&\beta_{p,X}^3
\left\|\sum_a\varepsilon_a g_a(x_{<k_a},y_{k_a})
\right\|_{L^{p}_{\varepsilon,x,y}}\\
&=&\beta_{p,X}^3\left\|\sum_k D_k^{\varepsilon\prime}\right\|_{L^{p}_{\varepsilon,x,y}}
\overset{\eqref{eq:umd-tangent-decoupling}}{\le}\beta_{p,X}^4\left\|\sum_kD_k^\varepsilon\right\|_{L^{p}_{\varepsilon,x}}\\
&=&\beta_{p,X}^4\left\|\sum_a\varepsilon_a d_{k_a}h_a\right\|_{L^{p}_{\varepsilon,x}}.
\end{eqnarray*}
This proves \eqref{eq:R-adapted-terminal}.
\end{proof}

\begin{proof}[Proof of Theorem~\ref{thm:R-bound}]
By the definition of the lexicographic order,
\begin{equation}\label{eq:S-decomposition}
S_{k,j}
=
E_k-E_0+\sum_{i=1}^{j-1}\Delta_{k,i}.
\end{equation}
The vector-valued Stein inequality \cite[Theorem~4.2.23]{HNVWI}
implies that \(\{E_k:k\ge0\}\) is an \(\mathcal R\)-bounded family
on \(L^p(\Gm;X)\), with
\begin{align*}
    \mathcal R_p\left(\{E_k:k\ge0\}\right)\lesssim_p \beta_{p,X}.
\end{align*}
Since \(d_k=E_{k+1}-E_k\), the same inequality shows that
\(\{d_k:k\ge0\}\) is also \(\mathcal R\)-bounded.
Combining these facts with Lemma~\ref{lem:R-adapted-terminal}
and \eqref{eq:S-decomposition}, we obtain
\[
\begin{aligned}
\left\|
 \sum_{a=1}^N\varepsilon_aS_{k_a,j_a}h_a
\right\|_{L^p(\Omega_\varepsilon\times\Gm;X)}
&\le
\left\|
 \sum_{a=1}^N\varepsilon_a(E_{k_a}-E_0)h_a
\right\|_{L^p(\Omega_\varepsilon\times\Gm;X)}
\\
&\quad+
\left\|
 \sum_{a=1}^N\varepsilon_a
 \sum_{i=1}^{j_a-1}\Delta_{k_a,i}h_a
\right\|_{L^p(\Omega_\varepsilon\times\Gm;X)}
\\
&\lesssim_{p}\beta_{p,X}^5
\left\|
 \sum_{a=1}^N\varepsilon_ah_a
\right\|_{L^p(\Omega_\varepsilon\times\Gm;X)}.
\end{aligned}
\]
This proves \eqref{eq:R-bound}.

It remains to prove the converse. Assume that
\(\sup_km_k=\infty\) and that the family \(\{S_{k,j}\}\) is
\(\mathcal R\)-bounded. In particular,
\[
C:=\sup_{(k,j)\in\Lambda_{\mathbf m}}
\|S_{k,j}\|_{L^p(\Gm;X)\to L^p(\Gm;X)}<\infty.
\]
Choose a sequence of indices \(k_r\) such that \(m_{k_r}\to\infty\).
Let
\[
h(t):=\sum_{|n|\le L}x_ne^{2\pi int}
\]
be an \(X\)-valued trigonometric polynomial. For all sufficiently
large \(r\), we have \(m_{k_r}>2L+1\), and hence \(L+1<m_{k_r}\).
Define
\[
f_r(\omega):=h\left(\frac{\omega_{k_r}}{m_{k_r}}\right),
\qquad \omega=(\omega_0,\omega_1,\ldots)\in\Gm.
\]
Using the identity \(\overline{\psi_a}=\psi_{0\ominus a}\),
we obtain
\[
f_r=\sum_{n=0}^Lx_n\psi_{nM_{k_r}}
    +\sum_{n=1}^Lx_{-n}\psi_{(m_{k_r}-n)M_{k_r}}.
\]
Since \(m_{k_r}-n>L\) for \(1\le n\le L\),
\eqref{eq:S-decomposition} implies
\begin{align*}
(S_{k_r,L+1}+E_0)f_r(\omega)
&=\left(E_{k_r}+\sum_{i=1}^L\Delta_{k_r,i}\right)f_r(\omega)\\
&=\sum_{n=0}^Lx_n\psi_{nM_{k_r}}(\omega)
= (P_+^{\mathbb T}h)\left(\frac{\omega_{k_r}}{m_{k_r}}\right).
\end{align*}
The assumed bound \(C\) and the contractivity of \(E_0\)
therefore yield
\[
\left(\frac1{m_{k_r}}\sum_{s=0}^{m_{k_r}-1}
\left\|(P_+^{\mathbb T}h)(s/m_{k_r})\right\|_X^p\right)^{1/p}
\le(C+1)
\left(\frac1{m_{k_r}}\sum_{s=0}^{m_{k_r}-1}
\left\|h(s/m_{k_r})\right\|_X^p\right)^{1/p}.
\]
Since the functions appearing in these Riemann sums are continuous,
passing to the limit gives
\[
\|P_+^{\mathbb T}h\|_{L^p(\mathbb T;X)}
\le(C+1)\|h\|_{L^p(\mathbb T;X)}.
\]
By density, \(P_+^{\mathbb T}\) extends to a bounded operator on
\(L^p(\mathbb T;X)\). It follows from Remark~\ref{Riesz}
that \(X\) is UMD.
\end{proof}

\section{Proof of Proposition \ref{prop:R-bounded-implies-alpha}}\label{s9}

The argument of
Cl\'ement et al.~\cite[Theorem~5.1]{ClementDePagterSukochevWitvliet2000}
for bounded Vilenkin systems relies on the unconditionality of the fine-block decomposition,
which fails in general in the unbounded setting. We therefore replace this ingredient by the $\mathcal R$-boundedness of the terminal
block projections $\pi_{k,\ell}$, obtained from
Theorem~\ref{thm:R-bound} and the vector-valued Stein inequality.
 We first recall the definition of property \((\alpha)\).

\begin{definition}[Property \((\alpha)\)]\label{def:property-alpha}
Fix \(1\le p<\infty\). A Banach space \(X\) has \emph{property \((\alpha)\)} if there is a constant
\(\alpha_{p,X}<\infty\) such that, for every \(N\ge1\), every finite family
\((x_{ij})_{1\le i,j\le N}\subset X\), and every scalar matrix
\((a_{ij})\) with \(|a_{ij}|\le1\),
\begin{equation}\label{eq:property-alpha}
\left\|\sum_{i,j=1}^N a_{ij}\varepsilon_i\varepsilon'_j x_{ij}\right\|_{L^p(\Omega_\varepsilon\times\Omega_{\varepsilon'};X)}
\le \alpha_{p,X}
\left\|\sum_{i,j=1}^N \varepsilon_i\varepsilon'_j x_{ij}\right\|_{L^p(\Omega_\varepsilon\times\Omega_{\varepsilon'};X)}.
\end{equation}
Here \((\varepsilon_i)_i\) and \((\varepsilon'_j)_j\) are independent
Rademacher sequences. By Kahane's inequality, this condition is independent
of the choice of \(p\in[1,\infty)\).
\end{definition}

We shall use the following fact proved in \cite{GirardiWeis2003}.
\begin{lemma}\label{Bochner2}
Let $X$ be a Banach space with property $(\alpha).$ For any $1\le p<\infty$ and any measure space $(S,\mu),$ the Bochner space $L^p(S;X)$ also has property $(\alpha)$, with $\alpha_{p,L^p(S;X)}\le \alpha_{p,X}.$
\end{lemma}

\begin{proof}[Proof of Proposition~\ref{prop:R-bounded-implies-alpha}]
We first prove sufficiency. Assume that $X$ is UMD and has
property $(\alpha).$ By Lemma~\ref{lem:paley-conjugation-identity}
and the complex contraction principle, it suffices to show that
\[\left\{\sum_{k=0}^{K}\pi_{k,n_k}d_k:n=\sum_{k=0}^Kn_kM_k\in\mathbb N\right\}\]
is \(\mathcal R\)-bounded on \(L^p(\Gm;X)\). Since
$$\pi_{k,0}=0,\qquad \pi_{k,\ell}=d_k-S_{k,m_k-\ell}+E_k-E_0,\qquad 1\le \ell<m_k,$$
the \(\mathcal R\)-boundedness of
$\{\pi_{k,\ell},k\ge 0,0\le \ell<m_k\}$ on $L^p(\Gm;X)$
follows from Theorem~\ref{thm:R-bound} and the vector-valued
Stein inequality \cite[Theorem~4.2.23]{HNVWI}.
Denote its $\mathcal R$-bound simply by $\mathcal R_p.$
Then $\mathcal R_p\lesssim_{p} \beta_{p,X}^5.$
By the UMD property, property $(\alpha)$ of $X$, and
Lemma~\ref{Bochner2}, for any $n^i=\sum_{k=0}^{K_i}n_k^iM_k$, $i=1,\dots,N$, and any
$f_i\in L^p(\Gm;X),$ we obtain the following estimates, with $n_k^i:=0$ for $k>K_i$:
\begin{align*}
\left\|\sum_{i=1}^N\varepsilon_i\sum_{k=0}^{K_i}\pi_{k,n^i_k}d_kf_i\right\|_{L^p(\Omega_\varepsilon\times\Gm;X)} &\le \beta_{p,X} \left\|\sum_{i=1}^N\sum_{k=0}^{K_i}\varepsilon_i\delta_k\pi_{k,n^i_k} d_kf_i\right\|_{L^p(\Omega_\varepsilon\times \Omega_\delta\times\Gm;X)}\\
&\le \alpha_{p,X}\beta_{p,X} \left\|\sum_{i=1}^N\sum_{k=0}^{\max K_i}\varepsilon_i\delta_k\rho_{i,k}\pi_{k,n^i_k} d_kf_i\right\|_{L^p(\Omega_\varepsilon\times \Omega_\delta\times \Omega_\rho\times\Gm;X)}\\
&\le \alpha_{p,X}\beta_{p,X} \mathcal R_p\left\|\sum_{i=1}^N\sum_{k=0}^{\max K_i}\varepsilon_i\delta_k\rho_{i,k}d_kf_i\right\|_{L^p(\Omega_\varepsilon\times \Omega_\delta\times \Omega_\rho\times\Gm;X)}\\
&\le \alpha_{p,X}^2\beta_{p,X} \mathcal R_p \left\|\sum_{k=0}^{\max K_i}\sum_{i=1}^N\varepsilon_i\delta_kd_kf_i\right\|_{L^p(\Omega_\varepsilon\times \Omega_\delta\times\Gm;X)}\\
&\le \alpha_{p,X}^2\beta_{p,X}^2 \mathcal R_p\left\|\sum_{i=1}^N\varepsilon_if_i\right\|_{L^p(\Omega_\varepsilon\times\Gm;X)},
\end{align*}
where $\{\delta_k\}_k$ and $\{\rho_{i,k}\}_{i,k}$ are mutually independent
families of Rademacher variables, also independent of $(\varepsilon_i)_i$.
This implies the $\mathcal R$-boundedness of
\[\left\{\sum_{k=0}^{K}\pi_{k,n_k}d_k:n=\sum_{k=0}^Kn_kM_k\in\mathbb N\right\}\]
and completes the proof of sufficiency.

We now prove necessity.
Assume that \(\{S_n:n\in\mathbb N\}\) is
\(\mathcal R\)-bounded on \(L^p(\Gm;X)\), and write
\[
\widetilde{\mathcal R}_p:=\mathcal R_p\bigl(\{S_n:n\ge0\}\bigr)<\infty,
\qquad S_0:=0.
\]
Since \(\|S_n\|\le \widetilde{\mathcal R}_p\),
Corollary~\ref{cor:umd-partial-sums-schauder} implies that \(X\)
is UMD. It remains to verify property \((\alpha)\).
Fix \(N\ge1\) and \(x_{ij}\in X\), \(1\le i,j\le N\). For
\(\eta^{(i)}=(\eta_{ij})_{j=1}^N\in\{0,1\}^N\), \(1\le i\le N\), set
\[
n_{\eta^{(i)}}
:=
\sum_{j=1}^N\eta_{ij}(m_{j-1}-1)M_{j-1},
\qquad
T_{\eta^{(i)}}
:=
\sum_{j=1}^N\eta_{ij}d_{j-1}.
\]
By \eqref{eq:weisz-partial-sum-identity},
\[
T_{\eta^{(i)}} f
=P_{B_{n_{\eta^{(i)}}}}f=
\overline{\psi_{n_{\eta^{(i)}}}}\,
S_{n_{\eta^{(i)}}}(\psi_{n_{\eta^{(i)}}}f).
\]
Therefore, by the complex contraction principle,
\begin{equation}\label{compare1}
    \left\|
\sum_{i=1}^N\varepsilon_iT_{\eta^{(i)}}f_i
\right\|_{L^p(\Omega_\varepsilon\times\Gm;X)}
\le
4\widetilde{\mathcal R}_p
\left\|
\sum_{i=1}^N\varepsilon_if_i
\right\|_{L^p(\Omega_\varepsilon\times\Gm;X)}
\end{equation}
for all \(\eta^{(i)}\in\{0,1\}^N\) and
\(f_i\in L^p(\Gm;X)\), \(i=1,\dots,N\).

For \(1\le j\le N\), define the martingale difference
\[
\rho_j(t):=
\begin{cases}
(-1)^{t_{j-1}},&m_{j-1}=2,\\[1mm]
\sin(2\pi t_{j-1}/m_{j-1}),&m_{j-1}\ge3.
\end{cases}
\]
We may assume that the family $\{\rho_j\}_{j=1}^N$ is independent
of $\{\varepsilon_j\}_{j=1}^N.$
By construction, the variables \(\rho_j\) are independent and
symmetric. Moreover, \(|\rho_j|\le1\), and
\[
\mathbb E|\rho_j|
\ge
\mathbb E|\rho_j|^2
\ge\frac12.
\]
Consequently, $\rho_j \overset{d}= \varepsilon_j |\rho_j|$,
and for every Banach space \(Y\) and every
\(y_1,\ldots,y_N\in Y\), we have
\begin{equation}\label{compare2}
\left\|\sum_{j=1}^N\rho_jy_j
\right\|_{L^p(\Gm;Y)}=\left\|\sum_{j=1}^N|\rho_j|\varepsilon_j'y_j
\right\|_{L^p(\Omega_{\varepsilon'}\times\Gm;Y)}
\simeq
\left\|\sum_{j=1}^N\varepsilon'_jy_j
\right\|_{L^p(\Omega_{\varepsilon'};Y)},
\end{equation}
where the lower bound follows from Jensen's inequality applied to
the expectation $\mathbb E_{(\rho_1,\dots,\rho_N)}$ and the real contraction principle.

For a fixed matrix $\eta=(\eta_{ij})\in\{0,1\}^{N\times N},$ set
\(
f_i:=\sum_{j=1}^N\rho_jx_{ij}.
\)
Since the variables $\rho_j,j=1,\dots,N$ form a martingale
difference sequence, we have
\(d_{j-1}\rho_k=\mathbf1_{\{j=k\}}\rho_k\), and hence
\[
T_{\eta^{(i)}}f_i
=
\sum_{j=1}^N\eta_{ij}\rho_jx_{ij}.
\]
Applying \eqref{compare2} with
\(Y=L^p(\Omega_\varepsilon;X)\) and using \eqref{compare1},
we obtain
\[
\begin{aligned}
\left\|
\sum_{i,j=1}^N\eta_{ij}\varepsilon_i\varepsilon'_jx_{ij}
\right\|_{L^p(\Omega_\varepsilon\times\Omega_{\varepsilon'};X)}
&\le 2 \left\|
\sum_{i=1}^N\varepsilon_iT_{\eta^{(i)}}f_i
\right\|_{L^p(\Omega_\varepsilon\times\Gm;X)}
\\
&\le
8\widetilde{\mathcal R}_p
\left\|
\sum_{i=1}^N\varepsilon_if_i
\right\|_{L^p(\Omega_\varepsilon\times\Gm;X)}
\\
&\le
8\widetilde{\mathcal R}_p
\left\|
\sum_{i,j=1}^N\varepsilon_i\varepsilon'_jx_{ij}
\right\|_{L^p(\Omega_\varepsilon\times \Omega_{\varepsilon'};X)}.
\end{aligned}
\]

For a sign matrix \(\theta\in \{-1,1\}^{N\times N}\), write
\(\theta_{ij}=2\eta_{ij}-1\). Then
\[
\left\|
\sum_{i,j=1}^N\theta_{ij}\varepsilon_i\varepsilon'_jx_{ij}
\right\|_{L^p(\Omega_\varepsilon\times \Omega_{\varepsilon'};X)}
\le
(16\widetilde{\mathcal R}_p+1)
\left\|
\sum_{i,j=1}^N\varepsilon_i\varepsilon'_jx_{ij}
\right\|_{L^p(\Omega_\varepsilon\times\Omega_{\varepsilon'};X)}.
\]
By convexity, for every real matrix
\((a_{ij})\) with \(|a_{ij}|\le1\), we have
\begin{equation}\label{real matrix}
    \left\|
\sum_{i,j=1}^Na_{ij}\varepsilon_i\varepsilon'_jx_{ij}
\right\|_{L^p(\Omega_\varepsilon\times \Omega_{\varepsilon'};X)}
\le
(16\widetilde{\mathcal R}_p+1)
\left\|
\sum_{i,j=1}^N\varepsilon_i\varepsilon'_jx_{ij}
\right\|_{L^p(\Omega_\varepsilon\times \Omega_{\varepsilon'};X)}.
\end{equation}
For every complex matrix
\((a_{ij})\) with \(|a_{ij}|\le1\), applying \eqref{real matrix}
to the real and imaginary parts of \((a_{ij})\) gives
\begin{equation}
    \left\|
\sum_{i,j=1}^Na_{ij}\varepsilon_i\varepsilon'_jx_{ij}
\right\|_{L^p(\Omega_\varepsilon\times \Omega_{\varepsilon'};X)}
\le
(32\widetilde{\mathcal R}_p+2)
\left\|
\sum_{i,j=1}^N\varepsilon_i\varepsilon'_jx_{ij}
\right\|_{L^p(\Omega_\varepsilon\times \Omega_{\varepsilon'};X)}.
\end{equation}
Thus \(X\) has property \((\alpha)\), and the proof is complete.
\end{proof}

\medskip

\noindent\textbf{Acknowledgements.}
This work was partially supported by the National Natural Science
Foundation of China (grant nos.~12071355, 12325105, 12031004,
W2441002, and W2611005).

\noindent\textbf{AI statement.}
The authors acknowledge the use of AI tools, including ChatGPT and
Rethlas, for language polishing, \LaTeX{} editing, and exploratory
mathematical discussions during the preparation of this manuscript.
Some ideas used in developing the proofs arose during interactions with
GPT-5.6 Sol and Rethlas. All AI-generated suggestions were subsequently
examined, reformulated, and independently verified by the authors, who
take full responsibility for the mathematical content of the manuscript.



\raggedbottom




\end{document}